\documentclass[12pt]{article}
\usepackage[T2A]{fontenc}
\usepackage[cp1251]{inputenc}
\usepackage[english]{babel}
\usepackage{amsmath,amsfonts,amssymb}
\usepackage[mathscr]{eucal}
\newcounter{zacountsec}
\newcommand{\eh}{\hfill}\newlength{\sperr}

\newcommand{\Section}[1]{{\stepcounter{zacountsec}\vspace{3mm}%
\hspace*{18mm}\normalsize\bf\arabic{zacountsec}.\parbox[t]{150mm}{#1}}}

{\end{minipage}}
\newenvironment{thm}[2]{\begin{sloppypar}%
{#1 #2.}\em{}}%
{\end{sloppypar}}

\newcommand{\proof}{\hspace*{9mm}{\settowidth{\sperr}{\rm  Proof
}%
\parbox[t]{1.3\sperr}{\rm  P\eh r\eh o\eh o\eh f} }}

\newcounter{zalit}

{\end{list}\end{small}}
\newlength{\addro}\newlength{\addrt}
{\end{minipage}}

\def\L2{{L_2(\Gamma)}}
\def\e2{{e_2(g,\,\Gamma)}}

\makeatletter
\renewcommand*{\@biblabel}[1]{\hfill#1.}
\makeatother

\usepackage{mathtools}
\usepackage{makecell}
\usepackage{tikz}
\usepackage{dsfont}
\usepackage{secdot}
\usepackage{longtable}
\usepackage{booktabs}       
\usepackage{multirow}

\sectiondot{section}
\sectiondot{subsection}
\usepackage{titlesec}

\titleformat{\section}
 {\centering\large\MakeUppercase}{\thesection.}{1em}{}

\titleformat{\subsection}{\centering\normalfont\normalsize\itshape}{\thesubsection.}{1em}{}

\usepackage[breaklinks=true,colorlinks=true,linkcolor=red, citecolor=blue, urlcolor=blue]{hyperref}%

\usepackage[backend=biber,bibencoding=utf8,sorting=none,style=gost-numeric,language=autobib,autolang=other,clearlang=true,sortcites=true,doi=false,isbn=false,date=year]{biblatex}

\usepackage{dsfont}

\usepackage{xcolor}
\usepackage{colortbl}
\usepackage{color}
\usepackage{graphicx}
\usepackage{multirow}
\usepackage{pifont}

\usepackage[font=small]{caption}
\usepackage[font=small]{subcaption}

\usepackage[algo2e]{algorithm2e} 
\usepackage{verbatim}
\usepackage{xspace} %

\usepackage{enumerate}
\usepackage{enumitem}

\newcommand{\dotprod}[2]{\langle #1,#2 \rangle}

\definecolor{PineGreen}{HTML}{008B72}
\newcommand{\greencheck}{\color{PineGreen}\ding{51}}
\newcommand{\redx}{\color{red} \ding{55}}

  \providecommand{\bb}{\mathbf{b}}

  \providecommand{\ee}{\mathbf{e}}

  \renewcommand{\gg}{\mathbf{g}}

  \providecommand{\nn}{\mathbf{n}}

  \providecommand{\uu}{\mathbf{u}}

  \usepackage{bm}

\usepackage[textsize=tiny]{todonotes}

\providecommand{\mycomment}[3]{\todo[caption={},size=footnotesize,color=#1!20, inline]{\textbf{#2: }#3}}%
\providecommand{\inlinecomment}[3]{%
  {\color{#1}#2: #3}}%
\newcommand\commenter[2]%
{%
  \expandafter\newcommand\csname i#1\endcsname[1]{\inlinecomment{#2}{#1}{##1}}
  \expandafter\newcommand\csname #1\endcsname[1]{\mycomment{#2}{#1}{##1}}
}

\newcommand{\circledOne}{\text{\ding{172}}}
\newcommand{\circledTwo}{\text{\ding{173}}}
\newcommand{\circledThree}{\text{\ding{174}}}
\newcommand{\circledFour}{\text{\ding{175}}}

\usepackage{url}

\newcolumntype{g}{>{\columncolor[HTML]{D9FFAD}} c}
\newcolumntype{r}{>{\columncolor[HTML]{ffb7b2}} c}
\newcolumntype{b}{>{\columncolor[HTML]{b8e0ff}} c}

\newtheorem{theorem}{Theorem}
\newtheorem{lemma}{Lemma}

\newtheorem{definition}{Definition}
\newtheorem{remark}{Remark}

\newtheorem{assumption}{Assumption}

\commenter{AL}{blue}
\newcommand{\al}[1]{{\color{blue}#1}}

\newcommand{\newal}[1]{{\color{black}#1}}

\usepackage{algorithm}
\usepackage{algorithmic}

\usepackage[symbol]{footmisc}

\begin{document}
\vspace*{0mm}

\begin{center}
\large\bf HIGHLY SMOOTH ZEROTH-ORDER METHODS FOR SOLVING \\ \vspace*{1mm}
OPTIMIZATION PROBLEMS UNDER THE PL CONDITION\footnote{The research was supported by Russian Science Foundation (project No. 21-71- 30005), https://rscf.ru/en/project/21-71-30005/.}
\end{center} 

\vspace*{1mm}
\begin{center}
{\large\rm\bf\copyright\,\,2023 \,\,\,\,\,\,\,  
A.\,V.~Gasnikov$^{a, b, c}$,
A.\,V.~Lobanov$^{a, c,*}$\footnote{The main contribution to the article belongs to Aleksandr Lobanov~\textless{}lobbsasha@mail.ru\textgreater{}. According to the rules of the journal, the authors of the article are arranged in \textbf{alphabetical} order.} and
F.\,S.~Stonyakin$^{a, d}$}
\\
\vspace*{3mm}
{\it $^{a}$ 
141700 Dolgoprudny, Institutskiy per., 9, Moscow Institute of Physics and Technology, Russia \\
$^{b}$ 
121205 Moscow, B. Boulevard 30, bld. 1, Skolkovo Institute of Science and Technology, Russia \\
$^{c}$ 
125047 Moscow, A. Solzhenitsyn st., 25, Institute for System Programming of the RAS, Russia \\
$^{d}$ 
295007 Simferopol, Prospekt Vernadskogo 4, V.I. Vernadsky
Crimean Federal University, Russia \\
$^{*}$e--mail: lobbsasha@mail.ru 
} \\
\medskip
Received \,\, \,\,\,\,2023, revised \,\,\,\,\,2023, accepted \,\,\,\,\,\, 2023.
\vspace*{2mm}
\end{center}

\noindent
{\small {\bf Abstract} -- In this paper, we study the black box optimization problem under the Polyak--Lojasiewicz~(PL) condition, assuming that the objective function is not just smooth, but has higher smoothness. By using "kernel-based'' approximations instead of the exact gradient in \newal{the} Stochastic Gradient Descent method, we improve the best\newal{-}known results of convergence in the class of gradient-free algorithms solving problems under the PL condition. We generalize our results to the case where a zeroth-order oracle returns a function value at a point with some adversarial noise. We verify our theoretical results on the example of solving a system of nonlinear equations.
}

\vspace*{2mm}

\section{Introduction}\label{Section:Introduction}

The black box problem is a fundamental optimization problem when the objective function has only a zero-order oracle \cite{Rosenbrock_1960}. Recently, this problem has received significant attention in the setting of reinforcement learning \cite{Choromanski_2018, Mania_2018}, federated learning \cite{Patel_2022, Lobanov_2022}, \newal{and} deep learning \cite{Chen_2017, Gao_2018}. Particularly in applied problems of multi-armed bandit \cite{Shamir_2017, Lattimore_2021}, online optimization \cite{Agarwal_2010, Bach_2016, Tsybakov_2022}, huge-scale optimization~\cite{Bubeck_2015,Bogolubsky_2016}, and hyperparameter tuning \cite{Hernandez_2014, Nguyen_2022}. In addition, the black box problem arises if the information about the derivatives is too expensive or not available \cite{conn2009introduction}.

To solve such problems in the convex case, there are various techniques for developing optimal gradient-free algorithms based on first-order optimization algorithms, optimal by three criteria at once: oracle complexity, iteration complexity, and maximum level of admissible noise still allowing to guarantee certain accuracy \cite{Gasnikov_2022}. The main idea of which \newal{is} to calculate the gradient approximation instead of an exact gradient \cite{Wasan_2004}. \newal{For example, $L_1$, and $L_2$} randomized approximations \cite{duchi2015optimal, Nesterov_2017,  Shamir_2017, gasnikov2016gradient, Tsybakov_2022} and "kernel-based" approximation \cite{Polyak_1990, Bach_2016, Tsybakov_2020, Novitskii_2021} are usually used for the smooth case. Where~the~"kernel-based" approximation takes advantage of higher-order smoothness as opposed to the $L_2$~and~$L_1$~randomized approximations. \newal{For} the non-smooth case, there are smoothing techniques via~$L_1$~and~$L_2$~randomized approximation \cite{Gasnikov_ICML, Lobanov_2022}. 

In contrast to the convex case, where algorithms have been intensively appearing and being studied in theory and practice in recent years \cite[e.g.,][and see above]{Karimireddy_2018}, the analysis of the black box problem in the nonconvex case is just beginning to be actively studied \cite{Ghadimi_2013, Hajinezhad_2017, Liu_2018}. One such problem that \newal{is} frequently common in the literature lately is the smooth (nonconvex) optimization problem under the Polyak--Lojasiewicz condition \cite{Polyak_1963, Lojasiewicz_1963, Karimi_2016, Belkin_2021, Stonyakin_2022}. In this formulation of the problem, \cite{Stich_2020} studied the biased Stochastic Gradient Descent (SGD) method and also proposed~its~gradient-free~counterpart.

We focus on solving the smooth black-box optimization problem (in particular, the nonconvex optimization) under the Polyak--Lojasiewicz condition. We propose an optimal zero-order algorithm (see Algorithm \ref{algo: ZO-MB-SGD_algorithm}) based on Mini-batch SGD (added batching to biased Stochastic Gradient Descent~\cite{Stich_2020} to improve iteration complexity). Using kernel approximation as a technique for developing a gradient-free algorithm, we show in theory and in practice a significant improvement in convergence rate estimates. We provide an extended analysis of Algorithm \ref{algo: ZO-MB-SGD_algorithm}, namely we consider a (close to reality) stochastic formulation of the problem with adversarial noise \cite{Risteski_2016, Vasin_2021, Stepanov_2021}.

\begin{table*}
\begin{minipage}{\textwidth}
\caption{Comparison of convergence results of Zero-Order Mini-batch SGD method via Kernel-based approximation, "Kernel approx." (This work) with the existing counterpart via Gaussian approximation, "Gaus. approx."~\cite{Stich_2020}. Notation: $d$ = dimension, $\delta(x)$~=~adversarial deterministic noise, $\xi$ = adversarial stochastic noise, $\Delta$ = level noise, $\beta$ = smoothness order parameter.}
\label{tab:table_errorfloor} 
\centering
\resizebox{\linewidth}{!}{
\begin{tabular}{ccccccc}\toprule
\multirow{2}{*}{Case} & \multirow{2}{*}{Zero-order oracle} & \multicolumn{2}{c}{Error floor} & \multirow{2}{*}{Stochastic?} &  \multicolumn{2}{c}{Adversarial noise} \\  \cmidrule{3-4} \cmidrule{6-7}
  &  & Gaussian approx. &  Kernel approx.  & & ADN? & ASN? \\ \midrule

deterministic & $\Tilde{f}(x) = f(x) + \delta(x)$  &  $\mathcal{O} \left( d^2 \Delta  \right)$ & $\mathcal{O} \left( d^{2} \Delta^{\al{\frac{2(\beta-1)}{\beta}}} \right)$ &  \redx &  \greencheck & \redx \\ \cmidrule{1-7}
two-point  & $\Tilde{f}(x,\xi) = f(x, \xi) + \delta(x)$ &  $\mathcal{O} \left( d^2 \Delta \right)$ & $\mathcal{O} \left( d^{2} \Delta^{\al{\frac{2(\beta-1)}{\beta}}} \right)$ &  \greencheck & \greencheck & \redx \\ \cmidrule{1-7}
\multirow{2}{*}{one-point} & $\Tilde{f}(x, \xi) = f(x) + \xi$ & $\varepsilon$ & $\varepsilon$ & \greencheck  & \redx & \greencheck \\ \cmidrule{2-7}
 & $\Tilde{f}(x,\xi) = f(x) + \xi + \delta(x)$ & $\mathcal{O} \left( d^2 \Delta  \right)$ & $\mathcal{O} \left( d^{2} \Delta^{\al{\frac{2(\beta-1)}{\beta}}} \right)$ & \greencheck & \greencheck & \greencheck \\ \bottomrule
\end{tabular}
}
\end{minipage}
\end{table*}

\subsection{Contribution}\label{Subsection:Contribution}
\begin{itemize}
    \item We present a Zero-Order Mini-batch SGD method that significantly improves existing estimates of the convergence results (in particular, estimates of the error floor, see the "deterministic" case of the zero-order oracle (first line) in Table \ref{tab:table_errorfloor} and Table \ref{tab:table_compare}) in a class of gradient-free algorithms for solving smooth nonconvex problems under the PL condition.
    
    \item We provide an extended theoretical analysis of the Zero-Order Mini-batch SGD algorithm by considering a stochastic setting using a gradient approximation with one-point and two-point feedback, when the zero-order oracle \newal{is} corrupted by an bounded adversarial deterministic ("ADN")  and stochastic ("ASN") noise. We show that also in the stochastic setting a Zero-Order Mini-batch Stochastic Gradient Descent method is superior to the existing counterpart~\cite{Stich_2020} (see stochastic cases of the zero-order oracle $\Tilde{f}(x,\xi)$ in Table \ref{tab:table_errorfloor} and Table~\ref{tab:table_compare}).
    
    \item We empirically verify our theoretical results by comparing the Zero-Order Mini-batch Stochastic Gradient Descent method with the existing algorithm using a typical example for problems under the Polyak--Lojasiewicz condition: solving a system of nonlinear equations.
    
    \item We demonstrate the effectiveness of the randomized approximation in practice by comparing it to the Gaussian approximation from \cite{Stich_2020} and show the advantages of "kernel-based" approximation in experiments by additionally comparing it to the $L_2$ approximation.
\end{itemize}

\subsection{Paper organization}
This article has the following structure. In Section~\ref{Section:Related_Work} we provide related works. We describe the statement of the problem in Section~\ref{Section:Setup}. We present the optimal gradient-free Algorithm~\ref{algo: ZO-MB-SGD_algorithm} in Section~\ref{Section:Convergence_results}. \newal{In} Section~\ref{Section:Extended_analysis}, we extend our analysis of the Algorithm~\ref{algo: ZO-MB-SGD_algorithm} to a stochastic setting with different models of adversarial noise.  In Section~\ref{Section:Discussion}, we discuss the results. While in Section~\ref{Section:Experiments} we analyze the effectiveness of our approach on a practical experiment. Section~\ref{Section:Conclusions} concludes the paper. We give a detailed proof of the Theorems and Lemmas in the Appendix~(Supplementary~Materials).


\begin{table}
\begin{minipage}{\textwidth}
\caption{Comparison of iteration complexity \#$N$ oracle complexity, \#$T$ and maximum noise~level~\#$\Delta$ of the algorithm with the following approaches: Kernel and Gaussian approximations for~zero-order~oracle~cases \circledOne~:~$\Tilde{f}(x) = f(x) + \delta(x)$, \circledTwo~:~$\Tilde{f}(x,\xi) = f(x, \xi) + \delta(x)$, \circledThree~:~$\Tilde{f}(x, \xi) = f(x) + \xi$, \circledFour~:~$\Tilde{f}(x,\xi) = f(x)+\xi+\delta(x)$. \small In both approaches, the dependence of the estimates on the batch size $B$ is presented. Notation: $\beta$ = order of smoothness, $\varepsilon$~=~accuracy of the solution to the problem, $d$ = dimension, $\mu$ = Polyak--Lojasiewicz~condition~constant.}
\label{tab:table_compare} 
\centering
\resizebox{\linewidth}{!}{
\begin{tabular}{ccccccccc}\toprule
\multirow{2}{*}{ Case  } & \multirow{2}{*}{ Batch} & \multicolumn{3}{c}{Gaussian approximation} & \phantom{-} & \multicolumn{3}{c}{Kernel approximation}  \\  \cmidrule{3-5} \cmidrule{7-9}
 & Size & \# $N$ & \# $T$  & \#  $\Delta$  &\phantom{-} & \# $N$ & \# $T$ & \# $\Delta$ \\ \midrule
\multirow{2}{*}{\circledOne, \circledTwo, \circledFour} & $B \in [1, \beta^3 d]$ & $\Tilde{\mathcal{O}} \left( \frac{d}{B}  \mu^{-1}\right)$ & $\Tilde{\mathcal{O}} \left( d  \mu^{-1}\right)$  & $\Delta \leq \mu \varepsilon d^{-3/2}$ &   & $\Tilde{\mathcal{O}} \left( \frac{d}{B}  \mu^{-1}\right)$ & $\Tilde{\mathcal{O}} \left( d  \mu^{-1}\right)$& $\Delta \leq \mu \varepsilon d^{-3/2}$  \\ \cmidrule{2-9}
 & \cellcolor[HTML]{b8e0ff} $B > \beta^3 d$ & $\Tilde{\mathcal{O}} \left(\mu^{-1} \right)$  & \cellcolor[HTML]{ffb7b2} $\max \left\{\Tilde{\mathcal{O}} \left(\mu^{-1} B \right), \Tilde{\mathcal{O}} \left( \frac{d^{4} \Delta^2}{ \varepsilon^{2} \mu^{3}}\right) \right\}$  & \cellcolor[HTML]{ffb7b2} $\Delta \leq  \mu \varepsilon d^{-2}$ &   & $\Tilde{\mathcal{O}} \left(\mu^{-1} \right)$ & \cellcolor[HTML]{D9FFAD} $\max \left\{\Tilde{\mathcal{O}} \left(\mu^{-1} B \right), \Tilde{\mathcal{O}} \left( \frac{d^{2 + \frac{2}{\beta-1}}\Delta^2}{ \varepsilon^{\frac{\beta}{\beta-1}} \mu^{\frac{2\beta -1}{\beta-1}}}\right) \right\}$& \cellcolor[HTML]{D9FFAD}$\Delta \leq \frac{\left( \mu \varepsilon \right)^{\frac{\beta}{2(\beta - 1)}}}{d^{\frac{\beta}{\beta - 1}}}$  \\ \cmidrule{1-9}
 \multirow{2}{*}{\circledThree}& $B \in [1, \beta^3 d]$ & $\Tilde{\mathcal{O}} \left( \frac{d}{B}  \mu^{-1}\right)$ & $\Tilde{\mathcal{O}} \left( d  \mu^{-1}\right)$  & $\Delta \leq \mu \varepsilon d^{-1}$ &   & $\Tilde{\mathcal{O}} \left( \frac{d}{B}  \mu^{-1}\right)$ & $\Tilde{\mathcal{O}} \left( d  \mu^{-1}\right)$& $\Delta \leq \mu \varepsilon d^{-1}$  \\ \cmidrule{2-9}
 & \cellcolor[HTML]{b8e0ff}  $B > \beta^3 d$ & $\Tilde{\mathcal{O}} \left(\mu^{-1} \right)$  & \cellcolor[HTML]{ffb7b2}$\max \left\{\Tilde{\mathcal{O}} \left(\mu^{-1} B \right), \Tilde{\mathcal{O}} \left( \frac{d^{4} \Delta^2}{ \varepsilon^{2} \mu^{3}}\right) \right\}$  & \cellcolor[HTML]{ffb7b2} $\Delta \leq  \mu \varepsilon d^{-2} B^{1/2}$ &   & $\Tilde{\mathcal{O}} \left(\mu^{-1} \right)$ & \cellcolor[HTML]{D9FFAD} $\max \left\{\Tilde{\mathcal{O}} \left(\mu^{-1} B \right), \Tilde{\mathcal{O}} \left( \frac{d^{2 + \frac{2}{\beta-1}}\Delta^2}{ \varepsilon^{\frac{\beta}{\beta-1}} \mu^{\frac{2\beta -1}{\beta-1}}}\right) \right\}$& \cellcolor[HTML]{D9FFAD} $\Delta \leq \frac{\left( \mu \varepsilon \right)^{\frac{\beta}{2(\beta - 1)}}}{d^{\frac{\beta}{\beta - 1}}} B^{1/2}$   \\ \bottomrule
\end{tabular}
}
\end{minipage}
\end{table}

\section{Related Work}\label{Section:Related_Work}

\paragraph{Polyak--Lojasiewicz condition.} The field of research on the nonconvex problem under the Polyak--Lojasiewicz condition can be traced back to \cite{Polyak_1963}, where it is shown that this condition is sufficient for the gradient descent to achieve \newal{a} global linear convergence rate. Recently, problems with Polyak--Lojasiewicz condition \newal{have been} actively investigated. For instance, \cite{Stonyakin_2022} proposed a new formulation of the problem under the Polyak--Lojasiewicz condition, including a compromise and an early stopping rule to guarantee its achievement. \newal{Already in} the new statement of the problem, \cite{Kuruzov_2022} proposed a fully adaptive gradient algorithm (with respect to the Lipschitz constant of the gradient and the noise level in the gradient) for solving problems with this condition. However, in this work we consider the standard statement of the problem under the Polyak--Lojasiewicz condition \cite[e.g. see][]{Karimi_2016, Yue_2022}. Also\newal{,} it is shown in \cite{Yue_2022} that non-accelerated algorithms are optimal for smooth problems under the Polyak--Lojasiewicz condition. Therefore, in this paper\newal{,} we consider non-accelerated first-order optimization algorithms as a base for developing optimal gradient-free~methods.

\paragraph{SGD type algorithms.} Many works \cite{Zhang_2004, Bottou_2010, Alistarh_2018, Wangni_2018, Stich_2019, Varre_2021} investigated Stochastic Gradient Descent in various settings. \newal{For} black-box problems, this algorithm and its modifications are the basis for developing gradient-free algorithms. For instance, \cite{Lobanov_2022} used the following first-order optimization algorithms as a basis for developing gradient-free algorithms in the federated learning setting: Minibatch Accelerated~SGD and Single-Machine Accelerated~SGD \cite[from ][]{Woodworth_2021}, and Federated Accelerated Stochastic Gradient Descent \cite[from ][]{Yuan_2020}, which are accelerated modifications of the SGD, namely the~AC-SA~\cite{Lan_2012}. In a nonconvex optimization problem with a Polyak--Lojasiewicz condition, \cite{Stich_2020}~studied a biased Stochastic Gradient Descent, where the oracle has access to biased and noisy gradient estimates, and showed that Stochastic Gradient Descent methods can generally converge only to the neighborhood of the solution. In this paper, we use the biased Stochastic Gradient Descent algorithm \cite{Stich_2020} as a basis for developing \newal{an} optimal zero-order method for solving a smooth nonconvex optimization problem, assuming that the Polyak--Lojasiewicz condition is satisfied.

\paragraph{Kernel approximation.} The works \cite{Polyak_1990, Bach_2016,Tsybakov_2020,Novitskii_2021} investigated and used kernel approximation as a technique for creating gradient-free algorithms. In the survey \cite{Gasnikov_2022} showed that the significant difference of this approximation from others is taking into account the advantages of \newal{the} high order of smoothness of the objective function, i.e. satisfying the \newal{H\"{o}lder condition}. All these works \cite{Bach_2016,Tsybakov_2020,Novitskii_2021} used the central finite difference (CFD) scheme instead of the forward finite difference scheme (FFD) in kernel approximation. It turns out there is an explanation, \cite{Scheinberg_2022} showed that in a \newal{smooth case,} one should use CFD, not FFD. In this paper, we use the kernel approximation since we assume that our function has higher smoothness, in contrast to \cite{Stich_2020}, where only smoothness is assumed.

\paragraph{Stochastic optimization.} Stochastic optimization problems have received special attention in the literature \cite{Stich_2019, Gorbunov_2020, Woodworth_2020, Mishchenko_2020, Cohen_2021, Gorbunov_2022}. For example, \cite{Gorbunov_2021} investigated a stochastic problem with non-sub-Gaussian (heavy-tailed) noise in the convex case, and \cite{Gorbunov_2022} investigated in the convex-concave case. For black-box problems, \cite{Tsybakov_2020, Novitskii_2021} studied a one-point gradient approximation with additive stochastic noise. \cite{Dvinskikh_2022} studied a two-point gradient approximation corrupted by an adversarial deterministic noise, and \cite{Lobanov_2022} studied a one-point gradient approximation with the same adversarial deterministic noise. In this paper\newal{,} we generalize our analysis to the stochastic optimization problem and consider gradient approximation with one-point and two-point feedback obtained via function realizations. We consider two cases of noise: noise as defined in \cite{Tsybakov_2020, Novitskii_2021, Tsybakov_2022} and noise as defined in \cite{Gasnikov_ICML, Dvinskikh_2022}. 

To the best of our knowledge for the moment there are lack of results around gradient-free methods for problems with Polyak--Lojasiewicz condition. The known results \cite{ehrhardt2018geometric,malik2019derivative,luo2019stochastic} are dominated by~\cite{Stich_2020} as we consider to be state of the art results.


\section{Setup}\label{Section:Setup}

    We study black-box optimization problems of the form:
    \begin{equation}
        \label{Init_Problem}
        f^* := \min_{x \in Q \subset \mathbb{R}^d} f(x),
    \end{equation}
    where $f: \mathbb{R}^d \rightarrow \mathbb{R}$ is \newal{a} function that we want to minimize over a closed convex subset $Q$ of $\mathbb{R}^d$. The problem \eqref{Init_Problem} is a general statement problem in the field of optimization. In order to narrow down the class of problems to be solved, let us define the function class using the following assumptions. 

    \subsection{Notation}\label{Subsection:Notation}

    We use $\dotprod{x}{y}:= \sum_{i=1}^{d} x_i y_i$ to denote standard inner product of $x,y \in \mathbb{R}^d$. We denote $l_p$~-~norms ($p~\geq~1$) in $\mathbb{R}^d$ as $\| x\|_p := \left( \sum_{i=1}^d |x_i|^p \right)^{1/p}$. \newal{Particularly,} for $l_2$-norm in $\mathbb{R}^d$ it follows $\| x \|_2~:=~\sqrt{\dotprod{x}{x}}$. We denote $l_p$-sphere as $S_p^d(r):=\left\{ x \in \mathbb{R}^d : \| x \|_p = r \right\}$. Operator $\mathbb{E}[\cdot]$ denotes full mathematical expectation. We \newal{use the} notation $\tilde{O} (\cdot)$ to hide logarithmic factors. 

    \subsection{Assumptions on the objective function}\label{Subsection:Assumptions_1}
    For all our theoretical results we assume that $f$ is not just smooth, but has a high order of smoothness.
    \begin{assumption}[Higher order smoothness] \label{Assumption_1}
        Let $l$ denote \newal{the} maximal integer number strictly less than~$\beta$. Let $\mathcal{F}_\beta(L)$ denote the set of all functions $f: \mathbb{R}^d \rightarrow \mathbb{R}$ which are differentiable $l$ times and for~all~$x,z \in Q$ the H\"{o}lder-type condition \newal{holds}:
        \begin{equation*}
            \left| f(z) - \sum_{0 \leq |n| \leq l} \frac{1}{n!} D^n f(x) (z-x)^n \right| \leq L_\beta \| z - x \|^\beta_2,
        \end{equation*}
        where $L_\beta>0$, the sum is over multi-index $n~=~(n_1, ..., n_d) \in \mathbb{N}^d$, we used the notation $n!~=~n_1! \cdots n_d!$, $|n| = n_1 + \cdots + n_d$, and $\forall v = (v_1, ..., v_d) \in \mathbb{R}^d$ we defined
        \begin{equation*}
            D^n f(x) v^n = \frac{\partial ^{|n|} f(x)}{\partial^{n_1}x_1 \cdots \partial^{n_d}x_d} v_1^{n_1} \cdots v_d^{n_d}. 
        \end{equation*}
    \end{assumption}
    Also we assume the Polyak--Lojasiewicz condition ($\mu$-PL) with parameter $\mu > 0$. 
    \begin{assumption}[$\mu$-PL]
    \label{Assumption_2}
        The function is differentiable and there exists constant $\mu > 0$ s.t. $\forall x\in\mathbb{R}^d$
        \begin{equation}\label{eq:Assumption_2}
            \| \nabla f(x) \|^2_2 \geq 2 \mu (f(x) - f^*).
        \end{equation}
    \end{assumption}
    
    Assumption \ref{Assumption_1} is quite common in the literature \cite[e.g.][]{Bach_2016,Tsybakov_2020,Novitskii_2021}. We introduced this assumption in order to take advantage of the properties of higher smoothness. As far as we know Assumption~\ref{Assumption_2} was introduced by \cite{Polyak_1963}. Functions satisfying this assumption are called gradient-dominated functions~\cite{Nesterov_2006}.

\subsection{Assumptions on the biased gradient oracle}\label{Subsection:Assumptions_2}

    We now formulate the standard assumptions about the gradient oracle for the biased SGD method \cite{Stich_2020}. To do this, we start by introducing the following definition.

    \begin{definition}[Biased Gradient Oracle]
    \label{Definition_1}
        A map $\mathbf{g}~:~\mathbb{R}^d~\times~\mathcal{D} \rightarrow \mathbb{R}^d$ s.t.
        \begin{equation}\label{eq:Definition_1}
            \mathbf{g}(x,\xi) = \nabla f(x) + \mathbf{b}(x) + \mathbf{n}(x, \xi)
        \end{equation}
        for a bias $\mathbf{b}: \mathbb{R}^d \rightarrow \mathbb{R}^d$ and zero-mean noise $\mathbf{n}: \mathbb{R}^d \times \mathcal{D} \rightarrow \mathbb{R}^d$, that is $\mathbb{E}_\xi \mathbf{n}(x, \xi) = 0, \forall x \in \mathbb{R}^d$.
    \end{definition}
    We assume that this gradient oracle has bias and noise, and they are bounded. 
    \begin{assumption}[($M,\sigma^2$)-bounded noise]
    \label{Assumption_3}
        There exists constants $M, \sigma^2 \geq 0$ such that $\forall x \in \mathbb{R}^d$
        \begin{equation}\label{eq:Assumption_3}
            \mathbb{E}_\xi \| \mathbf{n}(x, \xi) \|^2_2 \leq M \| \nabla f(x) + \mathbf{b}(x)\|^2_2 + \sigma^2.
        \end{equation}
    \end{assumption}
    \begin{assumption}[Bounded bias]
    \label{Assumption_4}
        There exists constants $0 \leq m < 1$, and $\zeta^2 \geq 0$ s.t. $\forall x \in \mathbb{R}^d$
        \begin{equation}\label{eq:Assumption_4}
            \| \mathbf{b}(x) \|^2_2 \leq m \| \nabla f(x) \|^2_2 + \zeta^2.
        \end{equation}
    \end{assumption}
    
    Many prior work in the context stochastic optimization often assumed the bounded noise and bounded bias. For example, the Assumption \ref{Assumption_3} is similar as in \cite{Stich_2019}. In the case without bias, Assumption~\ref{Assumption_3} is referred to as the strong growth condition \cite[e.g. ][]{Vaswani_2019}. \newal{Also} the Assumption \ref{Assumption_4} was used by~\cite{Hu_2021}. 
    
    The assumptions introduced in this subsection are necessary to use the following auxiliary lemma, since our \newal{approach to creating a gradient-free method} is based on \newal{the} SGD algorithm~\cite{Stich_2020} (see Appendix~\ref{Appendix:proof_Lemma}). 
    
    \begin{lemma}\label{Itog_theorem_all_in_text}
       Let $\{ x_k\}_{k \geq 0}$ denote the iterates of Algorithm SGD \cite{Stich_2020} with batching, function $f$ satisfy Assumptions \ref{Assumption_1}-\ref{Assumption_2} with $\beta = 2$, and the gradient oracle \eqref{eq:Definition_1} satisfy Assumptions \ref{Assumption_3}-\ref{Assumption_4}. Then there exists step size $\eta~\leq~ \frac{1}{(M+1) L_2}$ such that it \newal{holds} for all $ N \geq 0$ and an arbitrary batch size $B$
       \begin{equation*}
            \mathbb{E}[f(x_N)] - f^* \leq (1 - \eta \mu (1- m))^N (f(x_0) - f^*) + \frac{ \zeta^2}{2 \mu (1- m)}  + \frac{\eta L_2 \sigma^2}{2 B \mu (1- m)}  .
        \end{equation*}
    \end{lemma}



\section{Zero-Order Mini-batch SGD}\label{Section:Convergence_results}

    In this section\newal{,} we present our approach \newal{to solving problem} \eqref{Init_Problem} when the gradient oracle \eqref{eq:Definition_1} has no information about the derivatives of the objective function. The main idea of our approach is to create an optimal (on oracle complexity, iteration complexity, and maximum level of noise) gradient-free algorithm based on the first-order method (namely, Stochastic Gradient Descent). To begin, we introduce an approximation of the gradient oracle \eqref{eq:Definition_1} via a zero-order oracle (value of the objective function $f(x)$ with some adversarial deterministic noise~$\delta(x)$ such that $|\delta(x)| \leq \Delta$ and $\Delta > 0$):
    \begin{equation}\label{zero_order_oracle_1}
        \Tilde{f}(x) = f(x) + \delta(x),
    \end{equation} 
    which take advantage of the higher smoothness of the function \cite{Bach_2016,Tsybakov_2020, Novitskii_2021}. Such approximation is referred to as "kernel-based" approximation of gradient and has the following form
        \begin{equation}\label{zero_order_gradient}
            \mathbf{\Tilde{g}}(x,\ee) = d \frac{\Tilde{f}(x+\gamma r \ee) - \Tilde{f}(x - \gamma r \ee)}{2 \gamma} K(r) \ee,
        \end{equation}
        where $\gamma>0$ is smoothing parameter, $\ee$ is uniformly distributed in $S_2^d(1)$, $r$ is uniformly distributed in $[-1,1]$, $\ee$ and $r$ are independent, $K:~[-1,1]~\rightarrow~\mathbb{R}$ is a kernel function that satisfies
        \begin{gather*}
            \mathbb{E}[K(u)] = 0, \;\;\; \mathbb{E}[u K(u)] = 1, \;\;\;
            \mathbb{E}[u^j K(u)] = 0, \;\;\; j=2,...,l, \;\;\; \mathbb{E}[|u|^\beta |K(u)|] < \infty.
        \end{gather*}
    Now we can present Algorithm \ref{algo: ZO-MB-SGD_algorithm}. This algorithm is a modification of SGD method. Instead of calculating the first-order gradient oracle from Definition \ref{Definition_1}, we calculate an approximation of the gradient \eqref{zero_order_gradient}. Also, to achieve \newal{an} optimal iteration complexity, we add a batch size $B$.
    \begin{algorithm}[thb]
       \caption{Zero-Order Mini-batch Stochastic Gradient Descent (ZO-MB-SGD)}
       \label{algo: ZO-MB-SGD_algorithm}
    \begin{algorithmic}
       \STATE {\bfseries Input:} step size $\eta$, iteration number $N$, batch size $B$, Kernel $K: [-1, 1] \rightarrow \mathbb{R}$, smoothing parameter $\gamma$, $x_0~\in~\mathbb{R}^d$
       
       \FOR{$k=0$ {\bfseries to} $N-1$}
       \STATE \qquad {1.} ~~~Sample vectors $\ee_1, \ee_2 ..., \ee_B$ uniformly distributed on the unit sphere $S_2^d(1)$ and scalars 
       \STATE \qquad \textcolor{white}{1.} ~~~$r_1, r_2, ..., r_B$ uniformly distributed on the interval [-1, 1]
       \STATE \qquad {2.} ~~~Define $\Tilde{\gg}(x_k,\ee_i) =  d \frac{\Tilde{f}(x_k+\gamma r_i \ee_i) - \Tilde{f}(x_k - \gamma r_i \ee_i)}{2 \gamma} K(r_i) \ee_i$ using \eqref{zero_order_oracle_1}
       \STATE \qquad {3.} ~~~Calculate $\gg_k = \frac{1}{B} \sum_{i=1}^B \Tilde{\gg}(x_k,\ee_i)$ 
       \STATE \qquad {4.} ~~~$x_{k+1} \gets x_{k} - \eta \gg_k$
       \ENDFOR
       \STATE {\bfseries Return:} $x_{N}$
    \end{algorithmic}
    \end{algorithm}

    The next theorem presents the convergence result of ZO-MB-SGD method in terms of the expectation.
    \begin{theorem}\label{Theorem:1}
        Let function $f(x)$ satisfy Assumptions~\ref{Assumption_1}~-~\ref{Assumption_2} and the gradient approximation \eqref{zero_order_gradient} satisfy Assumptions~\ref{Assumption_3}~-~\ref{Assumption_4} then by choosing the step size $\eta~\leq~  \frac{1}{(M+1) L_2}$, there exists parameters
        \begin{equation*}
        \begin{split}
            & M \leq 6 \beta^3 d, \;\;\;\;\; \sigma^2 \leq \frac{3 \beta^3 d^2 L_2^2 \gamma^2}{4} + \frac{2 \beta^3 d^2 \Delta^2}{\gamma^2}, \quad m=0, \;\;\;\;\;\; \zeta^2 \leq \beta^2 d^2 \left( L_\beta^2 \gamma^{2(\beta-1)} + \frac{ \Delta^2}{\gamma^2} \right)
        \end{split}
        \end{equation*} 
         such that Algorithm \ref{algo: ZO-MB-SGD_algorithm} with smoothing parameter $\gamma~\leq~\mathcal{O} \left(  \Delta^{1/\beta} \right)$ achieves the following error floor
        \begin{equation*}\label{Error_Floor_Theorem_1}
           \mathbb{E}f(x_N) - f^*  = \max\left\{ \varepsilon, \mathcal{O} \left(d^{2} \Delta^{\frac{2(\beta-1)}{\beta}}\right) \right\},
        \end{equation*}
         where $f^*$ is the solution to problem \eqref{Init_Problem}, $L_2$ is the Lipschitz gradient constant with respect to 2-norm.
    \end{theorem}

    The convergence result of Theorem~\ref{Theorem:1} shows that Algorithm~\ref{algo: ZO-MB-SGD_algorithm} with gradient approximation \eqref{zero_order_gradient} achieves the error floor $\mathcal{O} \left(d^{2} \Delta^{\frac{2(\beta-1)}{\beta}}\right)$ with linear convergence rate. This asymptote arises due to the accumulation of adversarial noise in the bias $\bb(x)$. To achieve $\varepsilon$-accuracy of Problem \eqref{Init_Problem}, the maximum level of noise can be defined as $\Delta \leq \mathcal{O}\left( \varepsilon^{\frac{\beta}{2(\beta-1)}} d^{\frac{-\beta}{\beta-1}} \right)$. This result significantly improves the approach described in \cite{Stich_2020}. Namely, using the same concept of the zero-order oracle~\eqref{zero_order_oracle_1} Algorithm~\ref{algo: ZO-MB-SGD_algorithm} via Gaussian approximation achieves the error floor $ \mathcal{O} \left( d^{2} \Delta \right) $, respectively, with the following maximum level of noise $\Delta \leq \mathcal{O} \left( \varepsilon d^{-2} \right)$. For a detailed proof of Theorem~\ref{Theorem:1}, see Appendix~\ref{Appendix:Proof_Theorem_1}. Also see Appendix~\ref{Gaussian_Proof_Theorem1} for details on estimates for Gaussian smoothing approximation. 
    
    Note that the results of Theorem~\ref{Theorem:1} were obtained by using the deterministic setting of the zero-order oracle~\eqref{zero_order_oracle_1}. However, if we reformulate problem \eqref{Init_Problem} to a stochastic setting, specifying that $f(x) := \mathbb{E}_\xi f(x, \xi)$ and we use a gradient approximation \eqref{zero_order_gradient} with two-point feedback, replacing in the zero-order oracle \eqref{zero_order_oracle_1} the calculation of objective function value $f(x)$ with the calculation of objective function value on realizations $f(x,\xi)$, then it is not difficult to show that the convergence results of Theorem~\ref{Theorem:1} also hold for Algorithm~\ref{algo: ZO-MB-SGD_algorithm}, which uses the gradient approximation with two-point feedback, where the term two-point feedback implies that it is possible to get the value of the objective function at two points on one realization. See Appendix~\ref{Appendix:Proof_two_point} and \ref{Gaussian_Proof_Theorem2} for a detailed proof of the case where Algorithm~\ref{algo: ZO-MB-SGD_algorithm} uses a gradient approximation with two-point feedback, applying the "kernel-based" approximation and Gaussian approximation approaches, respectively. \newal{The} concepts of a stochastic zero-order oracle when two-point feedback is not available are explored in the Section~\ref{Section:Extended_analysis}.
    
    
    \begin{remark}
        It should be noted that Algorithm \ref{algo: ZO-MB-SGD_algorithm} has a linear convergence rate (see Theorem \ref{Theorem:1}). Then, since \newal{the} ZO-MB-SGD method uses a randomized scheme (randomization on the $L_2$ sphere), we~can obtain exact estimates of the large deviation probabilities using Markov's inequality \cite{Anikin_2015}:
        \begin{equation*}
            \mathcal{P}\left( f(x_{N(\varepsilon \omega)}) - f^* \geq \varepsilon \right) \leq \omega \frac{\mathbb{E}\left[ f(x_{N(\varepsilon \omega)} \right] - f^*}{\varepsilon \omega} \leq \omega.
        \end{equation*}
    \end{remark}

\section{Extended analysis of convergence via stochastic zero-order oracle}\label{Section:Extended_analysis}

    In this section, we continue our study of gradient-free methods for solving Problem \eqref{Init_Problem} with the $\mu$-PL condition (see Assumption \ref{Assumption_2}). We want to extend the analysis of Algorithm \ref{algo: ZO-MB-SGD_algorithm}. Specifically, for cases where the zero-order oracle has a stochastic setting $\Tilde{f}(x,\xi)$. To do this, we consider the concept~of gradient approximation when only one-point feedback is available \cite[see, e.g. ][]{Gasnikov_2017} with two types of adversarial noise: adversarial stochastic noise, which is quite common in the following works \cite{Tsybakov_2020, Novitskii_2021, Tsybakov_2022}, and adversarial deterministic noise, which is actively used in \cite{Dvinskikh_2022, Lobanov_2022, Gasnikov_ICML}. To introduce stochastic, let us rewrite the initial problem \eqref{Init_Problem}, where zero-order oracle returns an unbiased noisy stochastic function value~$f(x,\xi)$: 
    \begin{equation}
         \label{Stochastic_Problem}
        \min_{x \in Q \subset \mathbb{R}^d} \{f(x) := \mathbb{E}_\xi f(x,\xi)\}.
    \end{equation}

\subsection{Stochastic zero-order oracle with additive adversarial stochastic noise}\label{Subsection:Stochastic_noise}

     If for some reason two-point feedback is not available, you can use Zero-Order Mini-batch Stochastic Gradient Descent method (see Algorithm \ref{algo: ZO-MB-SGD_algorithm}) with one-point feedback. Then using the concept of additive stochastic noise, the "kernel-based" approximation of the gradient \eqref{zero_order_gradient} takes the following~form:
    \begin{equation}\label{zero_order_gradient_add_noise}
        \mathbf{\Tilde{g}}(x,\xi,\ee) = d \frac{f(x+\gamma r \ee) + \xi_1 - f(x - \gamma r \ee) - \xi_2}{2 \gamma} K(r) \ee,
    \end{equation}
    where the zero-order oracle in this case is defined in the following form:
    \begin{equation}
        \label{zero_order_oracle_3}
        \Tilde{f}(x,\xi) = f(x) + \xi,
    \end{equation}
    and $\xi_1 \neq \xi_2 $ are adversarial stochastic noises such that $\mathbb{E}[\xi_1^2] \leq \Tilde{\Delta}^2$ and $\mathbb{E}[\xi_2^2] \leq \Tilde{\Delta}^2$, $\Tilde{\Delta} \geq 0$, and the random variables $\xi_1$ and $\xi_2$ are independent from $\ee$ and $r$. Also, this concept does not require the assumption of zero-mean $\xi_1$ and $\xi_2$. It is enough that $\mathbb{E}[\xi_1 \ee] = 0$ and $\mathbb{E}[\xi_2 \ee] = 0$. The gradient approximation \eqref{zero_order_gradient_add_noise} has some similarities with the two-point gradient approximation (see discussion following Theorem~\ref{Theorem:1}), but it is a one-point gradient approximation because it is impossible to obtain the value of the objective function on one realization twice. 
    %
    The following theorem provides the convergence results of ZO-MB-SGD method (see Algorithm \ref{algo: ZO-MB-SGD_algorithm}) with the gradient approximation~\eqref{zero_order_gradient_add_noise}.
    \begin{theorem}\label{Theorem:add_noise}
        Let function $f(x)$ satisfy Assumptions~\ref{Assumption_1}~-~\ref{Assumption_2} and the gradient approximation \eqref{zero_order_gradient_add_noise} satisfy Assumptions~\ref{Assumption_3}~-~\ref{Assumption_4} then by choosing the step size $\eta~\leq~\frac{1}{(M+1) L_2}$ there exists parameters
        \begin{gather*}
             M \leq 18 \beta^3 d, \;\;\;\;\; \sigma^2 \leq \frac{3  d^2 L_2^2 \gamma^2}{4} +  \frac{2d^2 \Tilde{\Delta}^2}{\gamma^2}, \quad
             m=0, \;\;\;\;\;\; \zeta^2 \leq 8 \beta^2 d^2 L_\beta^2 \gamma^{2(\beta-1)} 
        \end{gather*}
        such that Algorithm \ref{algo: ZO-MB-SGD_algorithm} with approximation of the gradient~\eqref{zero_order_gradient_add_noise} has the following convergence~rate
        \begin{equation*}\label{Itaration_Theorem_2}
           \mathbb{E}f(x_N) - f^*  = \mathcal{O} \left( (1 - \eta \mu (1- m))^N (f(x_0) - f^*) \right), 
        \end{equation*}
         where $f^*$ is the solution to problem \eqref{Stochastic_Problem}.
    \end{theorem}
    
    The results of Theorem \ref{Theorem:add_noise} show that Algorithm \ref{algo: ZO-MB-SGD_algorithm} with gradient approximation \eqref{zero_order_gradient_add_noise} has a linear convergence rate. Also, in contrast to previous Theorem \ref{Theorem:1}, it does not have a pronounced asymptote. This effect is observed because the concept of zero-order oracle \eqref{zero_order_oracle_3} does not imply an accumulation~of adversarial stochastic noise in the bias, and also reduces the variance by the large batch~size~$B$. It is worth noting that the same convergence results of ZO-MB-SGD method (see Algorithm \ref{algo: ZO-MB-SGD_algorithm}) are obtained using the approach of Gaussian approximation with a zero-order oracle concept \eqref{zero_order_oracle_3}. The~proof~of Theorem \ref{Theorem:add_noise} for the kernel approximation approach can be found in more detail in~Appendix~\ref{Appendix:Proof_add_noise}. Also see Appendix~\ref{Gaussian_Proof_Theorem3} for a detailed proof of the convergence rate of Algorithm \ref{algo: ZO-MB-SGD_algorithm}, using a~Gaussian gradient approximation with one-point feedback via the zero-order oracle concept~\eqref{zero_order_oracle_3}.
    

\subsection{Stochastic zero-order oracle with mixed adversarial noise}\label{Subsection:One_point_feedback}

    Another concept with one-point feedback is combining the gradient approximation \eqref{zero_order_gradient_add_noise} (from Subsection \ref{Subsection:Stochastic_noise}, adversarial stochastic noise) with adversarial deterministic noise (e.g., from zero-order oracle~\eqref{zero_order_oracle_1}). Then the gradient approximation \eqref{zero_order_gradient} with one-point feedback takes the following~form:
    \begin{equation}
        \label{zero_order_gradient_one_point}
        \mathbf{\Tilde{g}}(x,\xi,\ee) = d \frac{\Tilde{f}(x+\gamma r \ee, \xi_1) - \Tilde{f}(x - \gamma r \ee, \xi_2)}{2 \gamma} K(r) \ee,
    \end{equation}
    where the zero-order oracle
    \begin{equation}
        \label{zero_order_oracle_4}
        \Tilde{f}(x, \xi_1) = f(x) + \xi_1 + \delta(x),
    \end{equation}
    and $\xi_1 \neq \xi_2 $ are adversarial stochastic noises such that $\mathbb{E}[\xi_1^2] \leq \Tilde{\Delta}^2$ and $\mathbb{E}[\xi_2^2] \leq \Tilde{\Delta}^2$, $\Tilde{\Delta} \geq 0$, and $\delta(x)$~is~adversarial deterministic noise, $|\delta(x)| \leq \Delta$, $\Delta \geq 0$ is a level of noise. The random variables $\xi_1$ and $\xi_2$ are independent from $\ee$ and $r$. Also, this concept does not require the assumption of zero-mean $\xi_1$ and $\xi_2$. It is enough that $\mathbb{E}[\xi_1 \ee] = 0$ and $\mathbb{E}[\xi_2 \ee] = 0$. It is worth noting that the approximation of~gradient \eqref{zero_order_gradient_one_point} is also a gradient approximation with one-point feedback as in Subsection \ref{Subsection:Stochastic_noise}. This approximation implies calculating the value of objective function on different realizations $\xi_1 \neq \xi_2$. In the case when $\xi_1 = \xi_2$ we can say that the gradient approximations considered in this subsection and in the discussion of Theorem \ref{Theorem:1} are identical.  
    %
    The following theorem presents the convergence results of ZO-MB-SGD method with the gradient approximation \eqref{zero_order_gradient_one_point} via a zero-order oracle \eqref{zero_order_oracle_4}.
    \begin{theorem}\label{Theorem:combine}
        Let function $f(x)$ satisfy Assumptions~\ref{Assumption_1}~-~\ref{Assumption_2} and the gradient approximation \eqref{zero_order_gradient_one_point} satisfy Assumptions~\ref{Assumption_3}~-~\ref{Assumption_4} then by choosing the step size $\eta~\leq~\frac{1}{(M+1) L_2}$ there exists parameters
        \begin{align*}
             M  \leq \beta^3 d, \quad \quad   \sigma^2 \leq \beta^3  d^2  L_2^2 \gamma^2 + \frac{\beta^3 d^2 (\Delta^2 + \Tilde{\Delta}^2)}{\gamma^2} , \quad \quad
             m=0, \quad \quad  \zeta^2 \leq \beta^2 d^2 L_\beta^2 \gamma^{2(\beta-1)} + \frac{\beta^2 d^2 \Delta^2}{\gamma^2}   
        \end{align*}
        such that Algorithm~\ref{algo: ZO-MB-SGD_algorithm} with gradient approximation~\eqref{zero_order_gradient_one_point} and $\gamma~\leq~\mathcal{O} \left(  \Delta^{1/\beta} \right)$ achieves the following error floor
        \begin{equation*}
           \mathbb{E}f(x_N) - f^*  = \max\left\{ \varepsilon, \mathcal{O} \left(d^{2} \Delta^{\frac{2(\beta-1)}{\beta}}\right)\right\},
        \end{equation*}
         where $f^*$ is the solution to problem \eqref{Stochastic_Problem}.
    \end{theorem}
    
    The results of Theorem \ref{Theorem:combine} show that the Algorithm \ref{algo: ZO-MB-SGD_algorithm} with gradient approximation \eqref{zero_order_gradient_one_point}, which is corrupted by adversarial deterministic and stochastic noises, converges to the following asymptote $ \mathcal{O} \left(d^{2} \Delta^{\frac{2(\beta-1)}{\beta}}\right)$ with linear rate. This result can be restated: the maximum permissible level of adversarial deterministic noise to achieve $\varepsilon$-accuracy is $\Delta \leq \mathcal{O}\left( \varepsilon^{\frac{\beta}{2(\beta-1)}} d^{\frac{-\beta}{\beta-1}} \right) $. As a result, we indicate exactly adversarial deterministic noise, because the adversarial stochastic noise does not accumulate in the bias as well as in the variance (if the batch size is large enough). In contrast to stochastic noise, it is the adversarial deterministic noise that defines the level of asymptote (since it accumulates in the bias). Note that in this zero-order oracle concept, too, gradient approximation with "kernel-based" approach achieves a better error floor than Gaussian approach $ \mathcal{O} \left( d^{2} \Delta \right)$. See Appendix~\ref{Appendix:Proof_combine} and \ref{Gaussian_Proof_Theorem4} for a proof of convergence results through kernel and Gaussian approximation~approaches,~respectively.
    


\section{Improved estimates of Table \ref{tab:table_compare}}
In this section, we show how the estimates obtained in this paper on the oracle $T$ complexity, as well as on the maximum noise level $\Delta$, change using an improved analysis (presented in \cite{Akhavan_2023}) of the bias and second moment estimates of the Kernel approximation of the gradient, see e.g., \eqref{zero_order_gradient}. 

Chronologically speaking, this paper was submitted to arXiv in May 2023 and is the first paper in which a gradient-free algorithm (via Kernel approximation) is proposed for solving including non-convex optimization problems satisfying the Polyak--Lojasiewicz condition. However, we believe that it is not correct not to mention the work \cite{Akhavan_2023}, which appeared on arXiv in June 2023, in which the authors proposed an improved analysis for the Kernel approximation and showed an improved estimate on the total number of oracle calls. Despite the superiority in terms of oracle complexity, the results of this paper are independently interesting and partially state of the art, at least in terms of iteration complexity (since we achieve optimal estimates in the case $\beta = 2$, as well as the best we know in the case $\beta > 2$), as well as in terms of maximum noise (since we explicitly derive $\Delta>0$, at which we can still guarantee <<good>> convergence).   

But applying the improved estimates on the bias
\begin{equation*}
    \left\| \mathbb{E} \left[\mathbf{\Tilde{g}}(x,\xi,e) \right] - \nabla f(x) \right\|_2 \leq \kappa_\beta \frac{L}{(l-1)!} \cdot \frac{d}{d + \beta - 1} \tau^{\beta - 1},
\end{equation*}
and second moment
\begin{equation*}
    \mathbb{E}\left[ \| \mathbf{\Tilde{g}}(x,\xi,e) 
 \|_2^2 \right] \leq 4d \mathbb{E}\left[ \| \nabla f(x) \|_2^2 \right] + 4 d \kappa L^2 \tau^2 + \frac{\kappa d^2 \Tilde{\Delta}^2}{\tau^2}.
\end{equation*}
from \cite{Akhavan_2023} (here the authors consider the case of a zero-order oracle \eqref{zero_order_oracle_3} with stochastic noise $\mathbb{E}\left[\xi_1^2\right] \leq \Tilde{\Delta}^2$, which is generalized to deterministic noise $|\delta(x)| \leq \Delta$ (see, for example, Appendix \ref{Appendix:Proof_Theorem_1})) to the analysis of the gradient-free algorithm proposed in this paper we obtain the following results (see Table \ref{tab:table_compare_new}, changes are highlighted). It is not difficult to see that the changes are mainly in the dimensionality of the problem $d$, namely the problem dimensionality in the obtained estimates does not depend on the order of smoothness $\beta$. Thus, at this point it is safe to say that Table \ref{tab:table_compare_new} represents the state-of-the-art results. 

\begin{table}
\begin{minipage}{\textwidth}
\caption{Comparison of iteration complexity \#$N$ oracle complexity, \#$T$ and maximum noise~level~\#$\Delta$ of the algorithm with the following approaches: Kernel and Gaussian approximations for~zero-order~oracle~cases \circledOne~:~$\Tilde{f}(x) = f(x) + \delta(x)$, \circledTwo~:~$\Tilde{f}(x,\xi) = f(x, \xi) + \delta(x)$, \circledThree~:~$\Tilde{f}(x, \xi) = f(x) + \xi$, \circledFour~:~$\Tilde{f}(x,\xi) = f(x)+\xi+\delta(x)$. \small In both approaches, the dependence of the estimates on the batch size $B$ is presented. Notation: $\beta$ = order of smoothness, $\varepsilon$~=~accuracy of the solution to the problem, $d$ = dimension, $\mu$ = Polyak--Lojasiewicz~condition~constant. \textbf{Updated Table \ref{tab:table_compare} with new analysis from paper \cite{Akhavan_2023}}}
\label{tab:table_compare_new} 
\centering
\resizebox{\linewidth}{!}{
\begin{tabular}{ccccccccc}\toprule
\multirow{2}{*}{ Case  } & \multirow{2}{*}{ Batch} & \multicolumn{3}{c}{Gaussian approximation} & \phantom{-} & \multicolumn{3}{c}{Kernel approximation}  \\  \cmidrule{3-5} \cmidrule{7-9}
 & Size & \# $N$ & \# $T$  & \#  $\Delta$  &\phantom{-} & \# $N$ & \# $T$ & \# $\Delta$ \\ \midrule
\multirow{2}{*}{\circledOne, \circledTwo, \circledFour} & $B \in [1, \beta^3 d]$ & $\Tilde{\mathcal{O}} \left( \frac{d}{B}  \mu^{-1}\right)$ & $\Tilde{\mathcal{O}} \left( d  \mu^{-1}\right)$  & $\Delta \leq \mu \varepsilon \al{d^{-1}}$ &   & $\Tilde{\mathcal{O}} \left( \frac{d}{B}  \mu^{-1}\right)$ & $\Tilde{\mathcal{O}} \left( d  \mu^{-1}\right)$& $\Delta \leq \mu \varepsilon \al{d^{-1}}$  \\ \cmidrule{2-9}
 & $B > \beta^3 d$ & $\Tilde{\mathcal{O}} \left(\mu^{-1} \right)$  & $\max \left\{\Tilde{\mathcal{O}} \left(\mu^{-1} B \right), \Tilde{\mathcal{O}} \left( \frac{\al{d^{2}} \Delta^2}{ \varepsilon^{2} \mu^{3}}\right) \right\}$  &  $\Delta \leq  \mu \varepsilon \al{d^{-1}}$ &   & $\Tilde{\mathcal{O}} \left(\mu^{-1} \right)$ &  $\max \left\{\Tilde{\mathcal{O}} \left(\mu^{-1} B \right), \Tilde{\mathcal{O}} \left( \frac{\al{d^2} \Delta^2}{ \varepsilon^{\frac{\beta}{\beta-1}} \mu^{\frac{2\beta -1}{\beta-1}}}\right) \right\}$& $\Delta \leq \frac{\left( \mu \varepsilon \right)^{\frac{\beta}{2(\beta - 1)}}}{\al{d}}$  \\ \cmidrule{1-9}
 \multirow{2}{*}{\circledThree}& $B \in [1, \beta^3 d]$ & $\Tilde{\mathcal{O}} \left( \frac{d}{B}  \mu^{-1}\right)$ & $\Tilde{\mathcal{O}} \left( d  \mu^{-1}\right)$  & $\Delta \leq \mu \varepsilon \al{d^{-1/2}}$ &   & $\Tilde{\mathcal{O}} \left( \frac{d}{B}  \mu^{-1}\right)$ & $\Tilde{\mathcal{O}} \left( d  \mu^{-1}\right)$& $\Delta \leq \mu \varepsilon \al{d^{-1/2}}$  \\ \cmidrule{2-9}
 &   $B > \beta^3 d$ & $\Tilde{\mathcal{O}} \left(\mu^{-1} \right)$  & $\max \left\{\Tilde{\mathcal{O}} \left(\mu^{-1} B \right), \Tilde{\mathcal{O}} \left( \frac{\al{d^{2}} \Delta^2}{ \varepsilon^{2} \mu^{3}}\right) \right\}$  &  $\Delta \leq  \mu \varepsilon \al{d^{-1}} B^{1/2}$ &   & $\Tilde{\mathcal{O}} \left(\mu^{-1} \right)$ &  $\max \left\{\Tilde{\mathcal{O}} \left(\mu^{-1} B \right), \Tilde{\mathcal{O}} \left( \frac{d^{2} \Delta^2}{ \varepsilon^{\frac{\beta}{\beta-1}} \mu^{\frac{2\beta -1}{\beta-1}}}\right) \right\}$&  $\Delta \leq \frac{\left( \mu \varepsilon \right)^{\frac{\beta}{2(\beta - 1)}}}{\al{d}} B^{1/2}$   \\ \bottomrule
\end{tabular}
}
\end{minipage}
\end{table}

\section{Discussion}\label{Section:Discussion}
    To create optimal algorithms for three criteria: oracle complexity, iteration complexity, and maximum level of admissible noise still allowing to guarantee certain accuracy usually use accelerated batched algorithms as a base. This way, optimal iterative and oracle complexities are achieved. However, for a~smooth optimization problem with a PL condition \cite{Yue_2022} showed that the non-accelerated algorithms are optimal. That is why we chose the stochastic gradient descent method (the batched version), which is frequently used in almost all fields of machine learning. By doing so, we guaranteed the optimal oracle and iteration complexity. Also using known techniques from \cite{Gasnikov_2022}, we obtained an estimate for the maximum allowable adversarial noise level, which is the best among the estimates we know. Thus, the proposed Algorithm \ref{algo: ZO-MB-SGD_algorithm} can highly likely be considered optimal by three criteria at~once: the~iteration number~\cite{Yue_2022}, the oracle complexity~\cite{duchi2015optimal}, and the maximum adversarial noise~level~\cite{Risteski_2016}.
    
    Our theoretical results show that in the presence of adversarial noise, Algorithm \ref{algo: ZO-MB-SGD_algorithm} or its analog (e.g., from \cite{Stich_2020}) converges only to an asymptote (error floor). Since adaptive to adversarial noise is an important property of zero-order algorithm \cite[see, e.g.,][]{Bogolubsky_2016}, we considered two models of adversarial noise: deterministic and stochastic. From Theorems \ref{Theorem:add_noise} and \ref{Theorem:combine}\newal{,} we can see that adversarial stochastic noise does not accumulate in the bias. This is an essential difference between the two noise models. 
    
    The theoretical results show the advantage of our approach \newal{over the approach that used} approximation with forward finite difference via Gaussian smoothing from \cite{Stich_2020}. They studied a SGD with biased gradient in an optimization problem under PL condition. But there exists for the smooth case a randomized approximation \cite{Gasnikov_2022}, e.g., via $L_2$-randomization, which was not considered in \cite{Stich_2020}. A natural question arises: Is Algorithm \ref{algo: ZO-MB-SGD_algorithm} superior to the gradient-free counterpart of Algorithm with approximation of the gradient via $L_2$-randomization in optimization problems under the PL condition? After all, the only difference between these approaches is that our approach (kernel approximation) uses increased smoothness information. We explore this question in the experiments. And it turns out the answer to that question is positive (see more details Section \ref{Section:Experiments}). Thus, the Zero-Order Mini-batch Stochastic Gradient Descent method (see Algorithm \ref{algo: ZO-MB-SGD_algorithm}) is robust for solving the problem under the PL condition.
    

    \section{Experiments}\label{Section:Experiments}
    
    In this section, we focus on verifying whether our theoretical bounds are aligned with the numerical performance of the Zero-Order Mini-batch SGD method. In particular, we compare Algorithm \ref{algo: ZO-MB-SGD_algorithm} with the gradient-free counterpart from \cite{Stich_2020} which uses Gaussian smoothing approximation instead of the~exact gradient. In all tests we understand adversarial noise as a computational error~(mantissa).

    We consider a standard problem satisfying the Polyak--Lojasiewicz condition. Namely, the solution of a system of $p$~nonlinear equations \cite{Kuruzov_2022}. The optimization problem \eqref{Init_Problem} have the following form:
    \begin{equation*}
        \min_{x \in \mathbb{R}^d} f(x) := \| g(x) \|_2^2,
    \end{equation*}
    where $g(x) = 0$ is a system of $p$ nonlinear equations such that $p \leq d$,
    \begin{equation*}
        g(x) = C \sin (x) + D \cos (x) - b,
    \end{equation*}
    $x \in \mathbb{R}^d$, $C,D \in \mathbb{R}^{p \times d}$ and $b \in \mathbb{R}^p$. We use the weighted sums of the Legendre polynomials as the~kernel $K(r)$. For example, we have the following values for $\beta = \{1,2,3,4,5,6 \}$ \cite{Bach_2016}:
    \begin{align*}
        &K_\beta(r)  = 3r & & \beta = 1,2;\\
        &K_\beta(r)  =  \frac{15 r}{4} (5 - 7r^3) & & \beta = 3,4;\\
        &K_\beta(r)  =  \frac{195 r}{64} (99r^4 - 126r^2 + 35) & & \beta = 5,6.
    \end{align*}
    
    In Figure \ref{fig:Compare_Batch_size}\newal{,} we compare Algorithm \ref{algo: ZO-MB-SGD_algorithm} ("Kernel" approximation) with the gradient-free counterpart of \newal{the} algorithm from \cite{Stich_2020} ("Gaussian" approximation). We observe that Algorithm \ref{algo: ZO-MB-SGD_algorithm} significantly surpasses its counterpart in terms of convergence rate and error floor. We also see that when we increase batch size (e.g., from $B=1$ to $B = 10$) convergence neighborhood of the ZO-MB-SGD method decreases.

    \begin{figure}[!h]
        \centering
        \begin{minipage}{.46\textwidth}
          \centering
          \includegraphics[width=.9\linewidth]{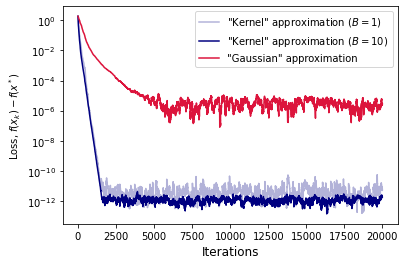}
          \captionof{figure}{\small Comparing Algorithm \ref{algo: ZO-MB-SGD_algorithm} with gradient-free Algorithm from \cite{Stich_2020} and effect of the parameter~$B$ (batch size) on the convergence neighborhood of Zero-Order Mini-batch Stochastic Gradient Descent. Here we optimize $f(x)$ with the parameters: $d=16$ (dimensional of problem), $p = 5$ (number of nonlinear equations), $\gamma = 0.01$ (smoothing parameter), $\eta~=~0.01$ (fixed step size), $B = \{1,10\}$ (batch size), $\beta = 3$ (order of smoothness).}
          \label{fig:Compare_Batch_size}
        \end{minipage}%
        \begin{minipage}{.05\textwidth}
        \textcolor{white}{hello}
        \end{minipage}%
        \begin{minipage}{.49\textwidth}
          \centering
          \includegraphics[width=1.0\linewidth]{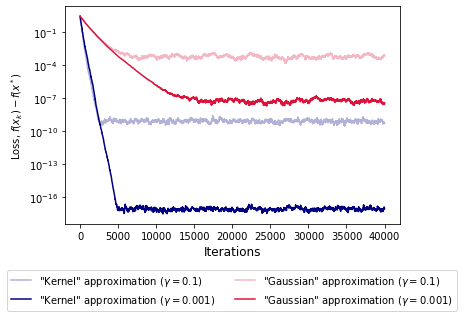}
          \captionof{figure}{\small Effect of the parameter $\gamma$ (smoothing parameter) on the error floor. Here we optimize $f(x)$ with parameters: $d=128$ (dimensional of problem), $p~=~16$ (number of nonlinear equations), $\gamma = \{ 0.1, 0.001 \}$ (smoothing parameter), $\eta~=~0.01$ (fixed step size), $B = 2$ (batch size), $\beta = 3$~(order~of~smoothness).}
          \label{fig:Effect_of_smoothing_parameter}
        \end{minipage}
    \end{figure}
    
    Figure \ref{fig:Effect_of_smoothing_parameter} shows the effect of the smoothing parameter $\gamma$ on the error floor. We can observe that the~parameter $\gamma$ directly affects the error floor. \newal{W}e also note that with different parameters $\gamma$ the Zero-Order Mini-batch SGD (see Algorithm \ref{algo: ZO-MB-SGD_algorithm}) retains its significant superiority~over~its~counterpart.
    
    To validate the robustness of our Algorithm \eqref{algo: ZO-MB-SGD_algorithm}, we now introduce the following gradient approximation via $L_2$ randomization which is well considered e.g. in \cite{Gasnikov_ICML}:
    \begin{equation}\label{L2_randomization}
        \tilde{\gg}(x,\ee)=d \frac{f(x + \gamma \ee) - f(x - \gamma \ee)}{2 \gamma} \ee,
    \end{equation}
    where $\ee$ is uniformly distributed on $S_2^d(1)$. Then Figure~\ref{fig:Effect_of_beta_parameter_without} demonstrates the effect of the smoothness order parameter on the convergence rate and the error floor, where "$L_2$ approximation" is a~gradient-free counterpart of algorithm \cite{Stich_2020} with \eqref{L2_randomization}. We see that $L_2$ approximation has the same convergence rate as "Gaussian" approximation, but has a better error floor (i.e. shows efficiency of the randomization). We also observe that the advantages of the high order of smoothness improve the error floor of Algorithm \ref{algo: ZO-MB-SGD_algorithm}. By comparing the $L_2$ approximation with Algorithm \ref{algo: ZO-MB-SGD_algorithm}, we see the efficiency of the "Kernel" (in terms of the error floor), which uses information about highly smoothness. See Appendix \ref{Appendix:Fig_4} for the effect of the parameters $d$ and $p$ on the convergence rate and the error floor.

    \begin{figure*}[!h]
    	\centering
    	\begin{minipage}{\textwidth}
    	    \centering
        	\hfill
        	\includegraphics[width=0.85\linewidth]{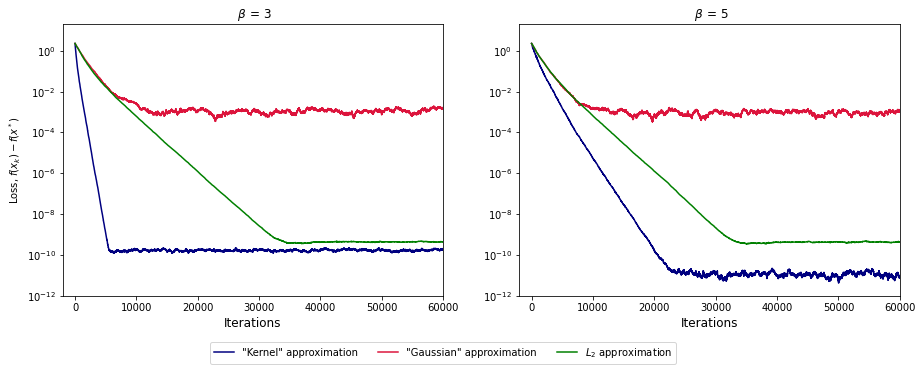}
        	\hfill \null
        	\caption{%
            Effect of the parameter $\beta$ on the convergence rate and the error floor. Here we optimize $f(x)$ with the parameters: $d=256$ (dimensional of problem), $p = 32$ (number of nonlinear equations), $\gamma = 0.1$ (smoothing parameter), $\eta~=~0.01$ (fixed step size), $B = 10$ (batch size), $\beta = \{3, 5\}$ (order of smoothness).}%
        	\label{fig:Effect_of_beta_parameter_without}
    	\end{minipage}
    \end{figure*}
    

\section{Conclusions}\label{Section:Conclusions}

    We proposed a Zero-Order Minibatch SGD method for solving smooth nonconvex optimization problems under the Polyak--Lojasiewicz condition. We generalized our results to stochastic problems under adversarial noise, considering the possible options: when two-point feedback is available and when only one-point feedback is available. Our algorithm showed efficiency on the standard (for~problems under the Polyak--Lojasiewicz condition) example of solving a system of nonlinear~equations.



\begin{center}
{\large CONFLICT OF INTEREST}
\end{center}
The authors declare that they have no conflict if interest.

\newpage

\printbibliography

\newpage

\newpage
\appendix

\section{Analyzing effect of problem parameters on algorithms}\label{Appendix:Fig_4}

    In Figure \ref{fig:Effects_of_dimensionality} we show the effect of the parameters $d$ and $p$ on the convergence rate and error floor. We can observe the behavior of the algorithms via the following approaches of approximating the gradient oracle (see Definition \ref{Definition_1}): "Kernel" approximation, "Gaussian" approximation and~$L_2$~approximation at different parameters of the dimensional of the problem and the number of nonlinear equations. When increasing the number of nonlinear equations, we see that all algorithms worsen the convergence rate, but do not change the error floor. And when increasing the dimensional of the problem, we observe that only the "Kernel" approximation and the $L_2$ approximation improve the error floor (i.e. shows the efficiency of the randomization). We note that Algorithm \ref{algo: ZO-MB-SGD_algorithm} surpasses its counterparts in all the above cases (at different parameters), thereby confirming that the Algorithm \ref{algo: ZO-MB-SGD_algorithm} is robust for solving the problem under the Polyak--Lojasiewicz condition~in~general.
    
    \begin{figure}[!h]
    	\centering
    	\begin{minipage}{\textwidth}
    	    \centering
        	\hfill
        	\includegraphics[width=1.0\linewidth]{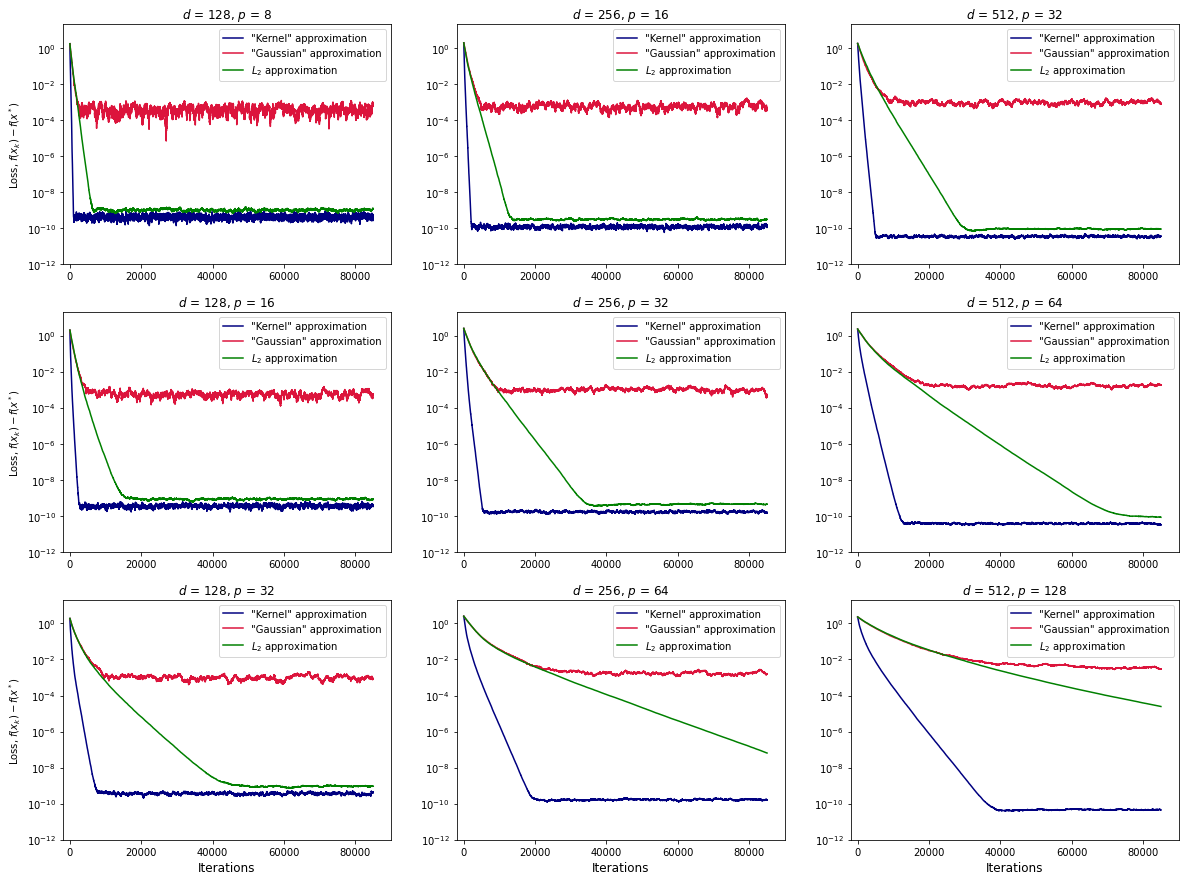}
        	\hfill \null
        	\caption{%
            Effect of the parameters $d$ (dimensional of problem) and $p$ (number of nonlinear equations) on the convergence rate and the error floor. Here we optimize $f(x)$ with the following parameters: $d=\{128, 256, 512\}$ (dimensional of problem), $p = \{8, 16, 32, 64, 128 \}$ (number of nonlinear equations), $\gamma = 0.1$ (smoothing parameter), $\eta~=~0.01$ (fixed step size), $B = 10$ (batch size), $\beta = 3$ (order of smoothness).}%
        	\label{fig:Effects_of_dimensionality}
    	\end{minipage}
    \end{figure}


\section{Auxiliary Facts and Results}

    In this section we list the auxiliary facts and results that we use several times in our proofs.
    
    \subsection{Squared norm of the sum} For all $a_1,...,a_n \in \mathbb{R}^d$, where $n=\{2,3\}$
    \begin{equation}
        \label{Squared_norm_of_the_sum}
        \|a_1 + ... + a_n \|_2^2 \leq n \| a_1 \|_2^2 + ... + n \| a_n \|_2^2.
    \end{equation}
    
    \subsection{Fact from concentration of the measure}
    Let $\ee$ is uniformly distributed on the Euclidean unit sphere, then, for $d \geq 8$, $\forall s \in \mathbb{R}^d$
    \begin{equation}
        \label{Concentration_measure}
        \mathbb{E}_\ee \left( \dotprod{s}{\ee}^2\right) \leq \frac{\| s \|_2^2}{d}. 
    \end{equation}
    
    \subsection{Gaussian smoothing. Upper bounds for the moments}
    \textbf{(Proved in Lemma 1, \cite{Nesterov_2017}).} Let $\uu \sim \mathcal{N}(0,1)$ is a random Gaussian vector, then we have
    \begin{align}
         \text{for } p \in [0,2]&:  &\mathbb{E}_u[\| u \|_2^p] &\leq d^{p/2}; 
         \label{2_moment_for_u} \\ 
         \text{for } p \geq 2&:  &\mathbb{E}_u[\| u \|_2^p] &\leq (p + d)^{p/2}. \label{p_moment_for_u}
    \end{align}
    
    \subsection{Fact from Gaussian approximation }
    \textbf{(Proved in Theorem 3, \cite{Nesterov_2017}).} Let $f$ is differentiable at $x \in \mathbb{R}^d$ and $\uu \sim \mathcal{N}(0,1)$ is normally distributed random Gaussian vector, then we have
    \begin{equation}
        \label{Fact_for_second_moment_Gaussian}
        \mathbb{E}_u [\dotprod{\nabla f(x)}{u}^2 \| u \|_2^2] \leq (d+4) \| \nabla f(x) \|^2_2.
    \end{equation}
    
    \subsection{Taylor expansion} Using the Taylor expansion we have
    \begin{equation}
        \label{Taylor_expansion_1}
        f(x+ \gamma r e) = f(x) + \dotprod{\nabla f(x)}{\gamma r e} + \sum_{2 \leq |n| \leq l} \frac{(r \gamma)^{|n|}}{n!} D^{(n)} f(x) e^n + R(\gamma r e),
    \end{equation}
    \newal{whereby} assumption 
    \begin{equation}\label{Taylor_expansion_2}
        |R(\gamma r e)| \leq L \| \gamma r e \|_2^\beta = L |r|^\beta \gamma^\beta.
    \end{equation}
    
    \subsection{Kernel property} If $e$ is uniformly distributed on unit sphere we have $\mathbb{E}[e e^{\text{T}}] = (1/d)I_{d \times d}$, where $I_{d \times d}$ is the identity matrix. Therefore, using the facts $\mathbb{E}[r K(r)] = 1$ and $\mathbb{E}[r^{|n|} K(r)] = 0$ for $2 \leq |n| \leq l$ we have
        \begin{equation}\label{Kernel_property}
            \mathbb{E}\left[ \frac{d}{\gamma} \left( \dotprod{\nabla f(x)}{\gamma r e} + \sum_{2 \leq |n| \leq l} \frac{(r \gamma)^{|n|}}{n!} D^{(n)} f(x) e^n \right) K(r) e  \right] = \nabla f(x).
        \end{equation}

    \subsection{Bounds of the Weighted Sum of Legendre Polynomials}
    Let $\kappa_\beta = \int |u|^\beta |K(u)| du$ and set $\kappa = \int K^2(u) du$. Then if $K$ be a weighted sum of Legendre polynomials, then it is proved in (see Appendix A.3, \cite{Bach_2016}) that $\kappa_\beta$ and $\kappa$ do not depend on $d$, they~depend~only~on~$\beta$, such that for $\beta \geq 1$: 
        \begin{equation}\label{eq_remark_1}
            \kappa_\beta \leq 2 \sqrt{2} (\beta-1),
        \end{equation}
        \begin{equation}\label{eq_remark_2}
            \kappa \leq 3 \beta^{3}.
        \end{equation}
    

    
    \section{The convergence rate of SGD with biased gradient}\label{Appendix:proof_Lemma}
    \begin{lemma}[Lemma \ref{Itog_theorem_all_in_text}]
       Let $\{ x_k\}_{k \geq 0}$ denote the iterates of Algorithm Mini-batch SGD, function $f$ satisfy Assumptions \ref{Assumption_1}-\ref{Assumption_2} with $\beta = 2$, and the gradient oracle \eqref{eq:Definition_1} satisfy Assumptions \ref{Assumption_3}-\ref{Assumption_4}. Then there exists step size $\eta~\leq~ \frac{1}{(M+1) L_2}$ such that it hold  $\forall N \geq 0$ and an arbitrary batch size $B$
       \begin{equation}
       \label{Itog_theorem_all}
            \mathbb{E}[f(x_N)] - f^* \leq (1 - \eta \mu (1- m))^N (f(x_0) - f^*) + \frac{ \zeta^2}{2 \mu (1- m)}  + \frac{\eta L_2 \sigma^2}{2 B \mu (1- m)}  .
        \end{equation}
       where $L_2$ is constant of the Lipschitz gradient such that $\| \nabla f(x) - \nabla f(y) \|_2 \leq L_2 \| x - y \|_2$.
    \end{lemma}
    \begin{proof}
         By the $L_2$-smoothness of $f$ and choice of the step size $\eta \leq \frac{1}{(M+1)L_2}$ we have
    \begin{align}
        \mathbb{E}[f(&x_{k+1})] \leq f(x_k) + \dotprod{\nabla f(x_k)}{x_{k+1} - x_k} + \frac{L_2}{2} \| x_{k+1} - x_k \|_2^2 
        \nonumber \\
        &\leq  f(x_k) - \eta \dotprod{\nabla f(x_k)}{\mathbb{E}[\gg_k]} + \frac{\eta^2 L_2}{2} \left( \mathbb{E}\left[ \|  \gg_k - \mathbb{E}[\gg_k] \|_2^2 \right] + \mathbb{E}\left[ \| \mathbb{E} [\gg_k] \|_2^2 \right] \right) 
        \nonumber \\
        & \overset{\eqref{eq:Definition_1}}{=} f(x_k) - \eta \dotprod{\nabla f(x_k)}{\nabla f(x_k) + \bb(x_k)} + \frac{\eta^2 L_2}{2} \left( \mathbb{E}\left[ \| \nn(x_k, \xi) \|_2^2 \right] + \mathbb{E}\left[ \| \nabla f(x_k) + \bb(x_k) \|_2^2 \right] \right) 
        \nonumber \\
        & \overset{\eqref{eq:Assumption_3}}{\leq} f(x_k) - \eta \dotprod{\nabla f(x_k)}{\nabla f(x_k) + \bb(x_k)} + \frac{\eta^2 L_2}{2} \left( (M+1) \mathbb{E}\left[ \|\nabla f(x_k) + \bb(x_k) \|_2^2 \right] + \sigma^2 \right) 
        \nonumber \\
        & = f(x_k) + \frac{\eta}{2} \left(\pm \|\nabla f(x_k) \|^2_2 -2 \dotprod{\nabla f(x_k)}{\nabla f(x_k) + \bb(x_k)} + \|\nabla f(x_k) + \bb(x_k) \|_2^2 \right) + \frac{\eta^2 L_2}{2}  \sigma^2 
        \nonumber \\
        & = f(x_k) + \frac{\eta}{2} \left(- \|\nabla f(x_k) \|^2_2 + \| \bb(x_k) \|_2^2 \right) + \frac{\eta^2 L_2}{2}  \sigma^2 
        \nonumber \\
        & \overset{\eqref{eq:Assumption_2}, \eqref{eq:Assumption_4}}{\leq} (1 - \eta \mu (1- m))(f(x_k) - f^*) + \frac{\eta \zeta^2}{2}  + \frac{\eta^2 L_2}{2}  \sigma^2 + f^* .
        \label{Poly_Itog_theorem_1}
    \end{align}
    Applying recursion to \eqref{Poly_Itog_theorem_1} and adding batching (with batch size $B$) we obtain
    \begin{equation*}
        \mathbb{E}[f(x_N)] - f^* \leq (1 - \eta \mu (1- m))^N (f(x_0) - f^*) + \frac{ \zeta^2}{2 \mu (1- m)}  + \frac{\eta L_2 \sigma^2}{2 B \mu (1- m)}  .
    \end{equation*}
    \end{proof}


\textbf{Further structure of the appendix} has close to block form. In particular, the first block (\newal{Appendices} \ref{Appendix:Proof_Theorem_1}-\ref{Appendix:Proof_combine}) describes in detail the obtaining results for the approach proposed in this article to develop a gradient-free algorithm. Namely, the approach that implied using the "kernel-based" approximation (Kernel approximation) instead of the gradient oracle \eqref{eq:Definition_1}. While the second block (\newal{Appendices}\ref{Gaussian_Proof_Theorem1}-\ref{Gaussian_Proof_Theorem4}) provides a detailed description of obtaining results for the competing approach of developing gradient-free algorithms. In particular, the approach that uses Gaussian approximation instead of the gradient oracle \eqref{eq:Definition_1}. Each block considers different cases of a zero-order oracle. For instance, Appendix \ref{Appendix:Proof_Theorem_1} and \ref{Gaussian_Proof_Theorem1} consider zero-order oracle with adversarial deterministic noise described in Section \ref{Section:Convergence_results}. The stochastic case with the same concept of adversarial noise in a zero-order oracle generating a gradient approximation with two-point feedback, which is briefly described in the discussion of the results of Theorem \ref{Theorem:1}, is discussed in Appendix \ref{Appendix:Proof_two_point} and \ref{Gaussian_Proof_Theorem2}. Also in Appendix~\ref{Appendix:Proof_add_noise}~and~\ref{Gaussian_Proof_Theorem3}, the case of a zero-order oracle with adversarial stochastic noise, described in Subsection~\ref{Subsection:Stochastic_noise}, is considered in detail. \newal{The final} case in each block (Appendix \ref{Appendix:Proof_combine} and \ref{Gaussian_Proof_Theorem4}) is \newal{a} zero-order oracle case, combining adversarial deterministic and stochastic noise considered in Subsection~\ref{Subsection:One_point_feedback}. Note that the last two considered cases of zero-order oracle generate a gradient approximation with one-point feedback (see~Section~\ref{Section:Extended_analysis}).

\section{Proof of Theorem \ref{Theorem:1}}\label{Appendix:Proof_Theorem_1}
   
\subsection{Kernel approximation}
    The "kernel-based" approximation of gradient \newal{has} the following form \eqref{zero_order_gradient}:
    \begin{equation*}
        \mathbf{\Tilde{g}}(x,\ee) = d \frac{\Tilde{f}(x+\gamma r \ee) - \Tilde{f}(x - \gamma r \ee)}{2 \gamma} K(r) \ee,
    \end{equation*}
    where $\Tilde{f}(x)$ is defined in \eqref{zero_order_oracle_1}.
\subsection{Bias square}
        By definition \eqref{zero_order_gradient} we have
        \begin{alignat}{2}
            \| \mathbb{E}[\mathbf{\Tilde{g}}(x,&\ee)] - \nabla f(x) \|_2 = \| \frac{d}{2 \gamma} \mathbb{E} \left[ \left( \Tilde{f}(x + \gamma r \ee) - \Tilde{f}(x - \gamma r \ee) \right) \ee K(r) \right] - \nabla f(x)\|_2 
            \nonumber\\ 
            &\overset{\eqref{zero_order_oracle_1}}{=} \phantom{5} \| \frac{d}{2 \gamma} \mathbb{E} \left[ \left( f(x + \gamma r \ee) - f(x - \gamma r \ee) + \delta(x + \gamma r \ee) - \delta(x - \gamma r \ee) \right) \ee K(r)\right] - \nabla f(x) \|_2 
            \nonumber\\ 
            &\overset{\eqref{Taylor_expansion_1}}{=} \phantom{5} \| \frac{d}{\gamma} \mathbb{E} [(\dotprod{\nabla f(x)}{\gamma r \ee} + \sum_{2 \leq |n| \leq l \; \text{odd}} \frac{(r \gamma)^{|n|}}{n!} D^{(n)} f(x) \ee^n + \frac{R(\gamma r \ee) - R(-\gamma r \ee)}{2}  
            \nonumber\\ 
            &\phantom{= \phantom{15} \| \frac{d}{\gamma} \mathbb{E} [(\dotprod{\nabla f(x)}{\gamma r \ee}}\; 
            +  \frac{\delta(x + \gamma r \ee) - \delta(x - \gamma r \ee)}{2} ) \ee K(r)] - \nabla f(x)\|_2
            \nonumber\\ 
            &\overset{\eqref{Kernel_property}}{=} \phantom{5} \| \frac{d}{2 \gamma} \mathbb{E} \left[ \left(R(\gamma r \ee) - R(-\gamma r \ee) + \delta(x + \gamma r \ee) - \delta(x - \gamma r \ee) \right) \ee K(r) \right] \| _2
            \nonumber\\
            & \leq \phantom{15} 
            \frac{d}{2 \gamma} \mathbb{E} \left[ |R(\gamma r \ee) - R(-\gamma r \ee) + \delta(x + \gamma r \ee) - \delta(x - \gamma r \ee)|  |K(r)| \right] 
            \nonumber\\
            &\overset{\eqref{Taylor_expansion_2}}{\leq} \phantom{5} \kappa_\beta d \left( L_\beta \gamma^{\beta-1} + \frac{\Delta}{\gamma} \right). \qquad /* \mbox{ the distribution of $\ee$ is symmetric  } */ \nonumber
        \end{alignat}
        Then, we can find bias square
        \begin{equation}\label{preparing_bias_square_Kernel}
           \| \mathbf{b}(x) \|_2^2 = \| \mathbb{E}[\mathbf{\Tilde{g}}(x,\ee)] - \nabla f(x) \|_2^2  \overset{\eqref{Squared_norm_of_the_sum}}{\leq} 2 \kappa_\beta^2 d^2 L_\beta^2 \gamma^{2(\beta-1)} + 2\frac{\kappa_\beta^2 d^2 \Delta^2}{\gamma^2}.
        \end{equation}
        
\subsection{Second moment}
        By definition \eqref{zero_order_gradient} we have
        \begin{eqnarray} 
            \mathbb{E} \left[ \| \mathbf{\Tilde{g}}(x,e) \|^2_2\right] &=& \frac{d^2}{4 \gamma^2} \mathbb{E} \left[\|\left (f(x+\gamma r e) - f(x - \gamma r e) + \delta(x + \gamma r e) - \delta(x - \gamma r e) \right) K(r) e\|^2_2 \right] 
            \nonumber\\
            &=& \frac{d^2}{4\gamma^2} \mathbb{E} \left[  \left(f(x+\gamma r e) - f(x - \gamma r e) + 2 \delta(x + \gamma r e) \right)^2 K^2(r) \| e \|_2^2 \right]  
            \nonumber\\
            &\overset{\eqref{Squared_norm_of_the_sum}}{\leq}& \frac{ d^2}{2 \gamma^2} \left( \mathbb{E} \left[ \left(f(x+\gamma r e) - f(x - \gamma r e) \right)^2 K^2(r)\right] + 4 \kappa \Delta^2 \right). \label{eq_1_for_theorem_1}
        \end{eqnarray}
        
        Using fact $\sqrt{\mathbb{E} \left[ \| e \|_q^4 \right]} \leq \min \left\{ q-1, 16 \ln d -8 \right\} d^{2/q - 1}$ to the first multiplier \eqref{eq_1_for_theorem_1} we get
        \begin{align}
            \frac{d^2}{2 \gamma^2} \mathbb{E}_{e}  [( f(x+\gamma r e) &- f(x - \gamma r e) )^2 ] = \frac{d^2}{2 \gamma^2} \mathbb{E}_{e} \left[ \left( (f(x+\gamma r e)  - f(x - \gamma r e) \pm f(x) \pm 2 \dotprod{\nabla f(x)}{\gamma r e} \right)^2 \right]  
            \nonumber \\
            & \overset{\eqref{Squared_norm_of_the_sum}}{\leq} \phantom{5}  \frac{3 d^2}{2 \gamma^2} \mathbb{E}_{ e} \left[ ( f(x+\gamma r e) - f(x) - \dotprod{\nabla f(x)}{\gamma r e} \right)^2
            \nonumber\\
            & \phantom{\leq \phantom{5} \frac{3 d^2}{2 \gamma^2} \mathbb{E}}  + \left( f(x-\gamma r e) - f(x) - \dotprod{\nabla f(x)}{-\gamma r e} )^2 + 4 \dotprod{\nabla f(x)}{\gamma r e}^2 \right]
            \nonumber \\ 
            & \leq \phantom{15} \frac{3 d^2}{2 \gamma^2} \mathbb{E}_{ e} \left[ \frac{L_2^2}{2} \| \gamma r e\|_2^4 + 4 \dotprod{\nabla f(x)}{\gamma r e}^2 \right] 
            \label{eq_2_for_theorem_1} \\ 
            &\overset{\eqref{Concentration_measure}}{\leq} \phantom{5} \frac{3 d^2}{2 \gamma^2} \left( \frac{L_2^2 \gamma^4}{2} \mathbb{E}_{ e} \left[\| e\|_2^4 \right]  + \frac{4 \gamma^2 \| \nabla f(x) \|_2^2}{d} \right)
            \nonumber \\
            & \leq \phantom{15} 6 d \| \nabla f(x) \|_2^2 + \frac{3  d^2 L_2^2 \gamma^2}{4}, \label{eq_3_for_theorem_1}
        \end{align}
        where \eqref{eq_2_for_theorem_1} we obtained by applying the property of the Lipschitz continuous gradient.
        
        By substituting \eqref{eq_3_for_theorem_1} into \eqref{eq_1_for_theorem_1} and using independence of $e$ and $r$ we obtain
        
        \begin{equation}\label{preparing_second_moment_Kernel}
            \mathbb{E} \left[ \| \mathbf{\Tilde{g}}(x,e) \|_2^2\right] \leq \kappa \left( 6 d \| \nabla f(x) \|_2^2 + \frac{3  d^2 L_2^2 \gamma^2}{4} +  \frac{2d^2 \Delta^2}{\gamma^2} \right). 
        \end{equation}
        \newpage
\subsection{Convergence rate}
        From the inequalities \eqref{preparing_bias_square_Kernel} and \eqref{preparing_second_moment_Kernel} we can conclude that Assumption \ref{Assumption_3} holds with the choice
        \begin{equation}\label{parameters_Ass:3_Theorem1_Kernel}
            \sigma^2 \overset{\eqref{eq_remark_2}}{=} \mathcal{O} \left(\beta^3  d^2  L_2^2 \gamma^2 + \frac{\beta^3 d^2 \Delta^2}{\gamma^2} \right), \;\;\;  M = \mathcal{O} \left( \beta^3 d \right).
        \end{equation}
        and that Assumption \ref{Assumption_4} also holds with the choice
        \begin{equation}\label{parameters_Ass:4_Theorem1_Kernel}
            m=0, \zeta^2 \overset{\eqref{eq_remark_1}}{=} \mathcal{O} \left( \beta^2 d^2 L_\beta^2 \gamma^{2(\beta-1)} + \frac{\beta^2 d^2 \Delta^2}{\gamma^2} \right) .
        \end{equation} 
        
        Now to find the asymptote to which the Algorithm \ref{algo: ZO-MB-SGD_algorithm} converges with the gradient approximation \eqref{zero_order_gradient}, substitute the parameters \eqref{parameters_Ass:3_Theorem1_Kernel}, \eqref{parameters_Ass:4_Theorem1_Kernel} into the second and third terms \eqref{Itog_theorem_all} with $\eta = \mathcal{O} \left( 1/M \right)$:
        \begin{equation*}
            \mathbb{E}[f(x_N)] - f^* \leq  \frac{ \zeta^2}{2 \mu (1- m)}  + \frac{\eta L_2 \sigma^2}{2 B \mu (1- m)} = \mathcal{O} \left( \beta^2 d^2 L_\beta^2 \gamma^{2(\beta-1)} + \frac{\beta^2 d^2 \Delta^2}{\gamma^2} + \frac{d L_2^2 \gamma^2}{B} + \frac{d \Delta^2}{B \gamma^2} \right).
        \end{equation*}
        
        Since $B$ can be taken as large, the first two terms are responsible for the asymptote. We find the optimal smoothing parameter~$\gamma$ that minimizes the first two terms:
        \begin{eqnarray}
            \mathbb{E}[f(x_N)] - f^* \leq \beta^2 d^2 \gamma^{2(\beta-1)} + \frac{\beta^2 d^2 \Delta^2}{\gamma^2} = \mathcal{O} \left(d^{2} \Delta^{\frac{2(\beta-1)}{\beta}}\right), \label{assimptota_theorem1_Kernel}
        \end{eqnarray}
        where $\gamma =  \Delta^{1/\beta}$  is optimal smoothing parameter.
        Then from \eqref{assimptota_theorem1_Kernel} we can find the maximum noise level, assuming that~$d^{2} \Delta \leq \varepsilon$, for $\varepsilon > 0$ then we have
        \begin{equation*}
             \Delta = \mathcal{O}\left( \varepsilon^{\frac{\beta}{2(\beta-1)}} d^{\frac{-\beta}{\beta-1}} \right) .
        \end{equation*}
        

\section{Proof of convergence of Algorithm \ref{algo: ZO-MB-SGD_algorithm} with two-point feedback}\label{Appendix:Proof_two_point}
   
\subsection{Kernel approximation}
    The "kernel-based" approximation of gradient has the following form:
    \begin{equation}\label{zero_order_gradient_two_point}
        \mathbf{\Tilde{g}}(x,\xi,\ee) = d \frac{\Tilde{f}(x+\gamma r \ee,\xi) - \Tilde{f}(x - \gamma r \ee,\xi)}{2 \gamma} K(r) \ee,
    \end{equation}
    where the zero-order oracle $\Tilde{f}(x,\xi)$ is defined as follows:
   \begin{equation}\label{zero_order_oracle_2}
       \Tilde{f}(x,\xi) = f(x,\xi) + \delta(x),
   \end{equation}
   $\delta(x)$ is adversarial deterministic noise, $|\delta(x)| \leq \Delta$ is a level of noise.
\subsection{Bias square}
        By definition \eqref{zero_order_gradient_two_point} we have
        \begin{align*}
            \| \mathbb{E}[\mathbf{\Tilde{g}}(x,&\xi,\ee)] - \nabla f(x) \|_2 = \| \frac{d}{2 \gamma} \mathbb{E} \left[ \left( \Tilde{f}(x + \gamma r \ee,\xi) - \Tilde{f}(x - \gamma r \ee,\xi) \right) \ee K(r) \right] - \nabla f(x)\|_2 
            \nonumber\\ 
            &\overset{\eqref{Stochastic_Problem},\eqref{zero_order_oracle_2}}{=} 
            \| \frac{d}{2 \gamma} \mathbb{E} \left[ \left( f(x + \gamma r \ee) - f(x - \gamma r \ee) + \delta(x + \gamma r \ee) - \delta(x - \gamma r \ee) \right) \ee K(r)\right] - \nabla f(x)\|_2 
            \nonumber\\ 
            &\overset{\eqref{Taylor_expansion_1}}{=} \phantom{88}
            \| \frac{d}{\gamma} \mathbb{E} [(\dotprod{\nabla f(x)}{\gamma r \ee} + \sum_{2 \leq |n| \leq l \; \text{odd}} \frac{(r \gamma)^{|n|}}{n!} D^{(n)} f(x) \ee^n + \frac{R(\gamma r \ee) - R(-\gamma r \ee)}{2}   
            \nonumber\\ 
            & \phantom{\overset{\eqref{Taylor_expansion_1}}{=} \phantom{88} \| \frac{d}{\gamma} \mathbb{E} [(\dotprod{\nabla f(x)}{\gamma r \ee}}\;
            + \frac{\delta(x + \gamma r \ee) - \delta(x - \gamma r \ee)}{2} ) \ee K(r)] - \nabla f(x)\|_2
            \nonumber\\ 
            &\overset{\eqref{Kernel_property}}{=} \phantom{88}
            \| \frac{d}{2 \gamma} \mathbb{E} \left[ \left(R(\gamma r \ee) - R(-\gamma r \ee) + \delta(x + \gamma r \ee) - \delta(x - \gamma r \ee) \right) \ee K(r) \right] \|_2 
            \nonumber\\
            & \leq \phantom{888}
            \frac{d}{2 \gamma} \mathbb{E} \left[ |R(\gamma r \ee) - R(-\gamma r \ee) + \delta(x + \gamma r \ee) - \delta(x - \gamma r \ee)|  |K(r)| \right] 
            \nonumber\\
            &\overset{\eqref{Taylor_expansion_2}}{\leq} \phantom{88}
            \kappa_\beta d \left( L_\beta \gamma^{\beta-1} + \frac{\Delta}{\gamma} \right). \qquad /* \mbox{ the distribution of $\ee$ is symmetric  } */ \nonumber
        \end{align*}
        Then, we can find bias square
        \begin{equation}\label{preparing_bias_square_Kernel_Theorem2}
           \| \mathbf{b}(x) \|_2^2 = \| \mathbb{E}[\mathbf{\Tilde{g}}(x,\xi,\ee)] - \nabla f(x) \|_2^2 \leq  \left(\kappa_\beta d L_\beta \gamma^{\beta-1} + \frac{\kappa_\beta d \Delta}{\gamma} \right)^2 \overset{\eqref{Squared_norm_of_the_sum}}{\leq} \frac{2\kappa_\beta^2 d^2}{\gamma^2} \left( L_\beta^2 \gamma^{2\beta} +  \Delta^2\right). 
        \end{equation}
        
\subsection{Second moment}
        By definition \eqref{zero_order_gradient_two_point} we have
        \begin{align} 
            \mathbb{E} \left[ \| \mathbf{\Tilde{g}}(x,\xi,e) \|^2_2\right] &= \phantom{88} \frac{d^2}{4 \gamma^2} \mathbb{E} \left[\|\left (f(x+\gamma r e,\xi) - f(x - \gamma r e,\xi) + \delta(x + \gamma r e) - \delta(x - \gamma r e) \right) K(r) e\|^2_2 \right] 
            \nonumber\\
            &= \phantom{88}
            \frac{d^2}{4\gamma^2} \mathbb{E} \left[  \left(f(x+\gamma r e,\xi) - f(x - \gamma r e,\xi) + 2 \delta(x + \gamma r e) \right)^2 K^2(r) \| e \|_2^2 \right]  
            \nonumber\\
            &\overset{\eqref{Squared_norm_of_the_sum}}{\leq} \phantom{8}
            \frac{ d^2}{2 \gamma^2} \left( \mathbb{E} \left[ \left(f(x+\gamma r e,\xi) - f(x - \gamma r e,\xi) \right)^2 K^2(r)\right] + 4 \kappa \Delta^2 \right).  \label{eq_1_for_theorem_2}
        \end{align}
        
        Using fact $\sqrt{\mathbb{E} \left[ \| e \|_q^4 \right]} \leq \min \left\{ q-1, 16 \ln d -8 \right\} d^{2/q - 1}$ to the first multiplier \eqref{eq_1_for_theorem_2} we get
        \begin{align}
            \frac{d^2}{2 \gamma^2} \mathbb{E}_{e}  [( f(x&+\gamma r e,\xi) - f(x - \gamma r e,\xi) )^2] 
            \nonumber \\
            & = \phantom{88}
            \frac{d^2}{2 \gamma^2} \mathbb{E}_{e} \left[ \left( (f(x+\gamma r e,\xi)  - f(x - \gamma r e,\xi) \pm f(x,\xi) \pm 2 \dotprod{\nabla f(x,\xi)}{\gamma r e} \right)^2 \right]  
            \nonumber \\
            & \overset{\eqref{Squared_norm_of_the_sum}}{\leq} \phantom{8}
            \frac{3 d^2}{2 \gamma^2} \mathbb{E}_{ e} \left[ ( f(x+\gamma r e,\xi) - f(x,\xi) - \dotprod{\nabla f(x,\xi)}{\gamma r e} \right)^2
            \nonumber\\
            & \phantom{\overset{\eqref{Squared_norm_of_the_sum}}{\leq}
            \frac{3 d^2}{2 \gamma^2} \mathbb{E}_{e}} + 
            \left( f(x-\gamma r e,\xi) - f(x,\xi) - \dotprod{\nabla f(x,\xi)}{-\gamma r e} )^2 + 4 \dotprod{\nabla f(x,\xi)}{\gamma r e}^2 \right]
            \nonumber \\ 
            &\leq  \phantom{88}
            \frac{3 d^2}{2 \gamma^2} \mathbb{E}_{ e} \left[ \frac{L_2^2}{2} \| \gamma r e\|_2^4 + 4 \dotprod{\nabla f(x,\xi)}{\gamma r e}^2 \right] 
            \label{eq_2_for_theorem_2} \\ 
            &\overset{\eqref{Concentration_measure}}{\leq} \phantom{8}
            \frac{3 d^2}{2 \gamma^2} \left( \frac{L_2^2 \gamma^4}{2} \mathbb{E}_{ e} \left[\| e\|_2^4 \right]  + \frac{4 \gamma^2 \| \nabla f(x,\xi) \|_2^2}{d} \right)
            \nonumber \\
            & \leq \phantom{88}
            6 d \| \nabla f(x,\xi) \|_2^2 + \frac{3  d^2 L_2^2 \gamma^2}{4}, \label{eq_3_for_theorem_2}
        \end{align}
        where \eqref{eq_2_for_theorem_2} we obtained by applying the property of the Lipschitz continuous gradient.
        
        By substituting \eqref{eq_3_for_theorem_2} into \eqref{eq_1_for_theorem_2} and using independence of $e$ and $r$ we obtain
        
        \begin{equation}\label{preparing_second_moment_Kernel_Theorem2}
            \mathbb{E} \left[ \| \mathbf{\Tilde{g}}(x,\xi,e) \|_2^2\right] \leq \kappa \left( 6 d \| \nabla f(x,\xi) \|_2^2 + \frac{3  d^2 L_2^2 \gamma^2}{4} +  \frac{2d^2 \Delta^2}{\gamma^2} \right). 
        \end{equation}
        
\subsection{Convergence rate}
        From the inequalities \eqref{preparing_bias_square_Kernel_Theorem2} and \eqref{preparing_second_moment_Kernel_Theorem2} we can conclude that Assumption \ref{Assumption_3} holds with the choice
        \begin{equation}\label{parameters_Ass:3_Theorem2_Kernel}
            \sigma^2 \overset{\eqref{eq_remark_2}}{=} \mathcal{O} \left(\beta^3  d^2  L_2^2 \gamma^2 + \frac{\beta^3 d^2 \Delta^2}{\gamma^2} \right), \;\;\;  M = \mathcal{O} \left( \beta^3 d \right).
        \end{equation}
        and that Assumption \ref{Assumption_4} also holds with the choice
        \begin{equation}\label{parameters_Ass:4_Theorem2_Kernel}
            m=0, \zeta^2 \overset{\eqref{eq_remark_1}}{=} \mathcal{O} \left( \beta^2 d^2 L_\beta^2 \gamma^{2(\beta-1)} + \frac{\beta^2 d^2 \Delta^2}{\gamma^2} \right) .
        \end{equation} 
        
        Now to find the asymptote to which the Algorithm \ref{algo: ZO-MB-SGD_algorithm} converges with the gradient approximation \eqref{zero_order_gradient_two_point}, substitute the parameters \eqref{parameters_Ass:3_Theorem2_Kernel}, \eqref{parameters_Ass:4_Theorem2_Kernel} into the second and third terms \eqref{Itog_theorem_all} with $\eta = \mathcal{O} \left( 1/M \right)$:
        \begin{equation*}
            \mathbb{E}[f(x_N)] - f^* \leq  \frac{ \zeta^2}{2 \mu (1- m)}  + \frac{\eta L_2 \sigma^2}{2 B \mu (1- m)} = \mathcal{O} \left( \beta^2 d^2 L_\beta^2 \gamma^{2(\beta-1)} + \frac{\beta^2 d^2 \Delta^2}{\gamma^2} + \frac{d L_2^2 \gamma^2}{B} + \frac{d \Delta^2}{B \gamma^2} \right).
        \end{equation*}
        
        Since $B$ can be taken as large, the first two terms are responsible for the asymptote. We find the optimal smoothing parameter~$\gamma$ that minimizes the first two terms:
        \begin{eqnarray}
             \mathbb{E}[f(x_N)] - f^* \leq \beta^2 d^2 \gamma^{2(\beta-1)} + \frac{\beta^2 d^2 \Delta^2}{\gamma^2} = \mathcal{O} \left(d^{2} \Delta^{\frac{2(\beta-1)}{\beta}}\right), \label{assimptota_theorem2_Kernel}
        \end{eqnarray}
        where $\gamma =  \Delta^{1/\beta}$  is optimal smoothing parameter.
        Then from \eqref{assimptota_theorem2_Kernel} we can find the maximum noise level, assuming that~$d^{2} \Delta \leq \varepsilon$, for $\varepsilon > 0$ then we have
        \begin{equation*}
             \Delta = \mathcal{O}\left( \varepsilon^{\frac{\beta}{2(\beta-1)}} d^{\frac{-\beta}{\beta-1}} \right) .
        \end{equation*}


\section{Proof of Theorem \ref{Theorem:add_noise}}\label{Appendix:Proof_add_noise}
   
\subsection{Kernel approximation}
    The "kernel-based" approximation of gradient has the following form \eqref{zero_order_gradient_add_noise}:
    \begin{equation*}
        \mathbf{\Tilde{g}}(x,\xi,\ee) = d \frac{\Tilde{f}(x+\gamma r \ee,\xi_1) - \Tilde{f}(x - \gamma r \ee,\xi_2)}{2 \gamma} K(r) \ee,
    \end{equation*}
    where $\Tilde{f}(x,\xi)$ is defined in \eqref{zero_order_oracle_3}.
\subsection{Bias square}
        By definition \eqref{zero_order_gradient_add_noise} we have
        \begin{align*}
            \| \mathbb{E}[\mathbf{\Tilde{g}}(x,\xi,\ee)] - \nabla f(x) \|_2 \phantom{88} &= \phantom{88}
            \| \frac{d}{2 \gamma} \mathbb{E} \left[ \left( \Tilde{f}(x + \gamma r \ee,\xi) - \Tilde{f}(x - \gamma r \ee,\xi) \right) \ee K(r) \right] - \nabla f(x)\|_2 
            \nonumber\\ 
            & \overset{\eqref{zero_order_oracle_3}}{=} \phantom{8}\,
            \| \frac{d}{2 \gamma} \mathbb{E} \left[ \left( f(x + \gamma r \ee) - f(x - \gamma r \ee) + \xi_1 - \xi_2 \right) \ee K(r)\right] - \nabla f(x)\|_2 
            \nonumber\\ 
            &\overset{\eqref{Taylor_expansion_1}}{=} \phantom{8}\,
            \| \frac{d}{\gamma} \mathbb{E} [(\dotprod{\nabla f(x)}{\gamma r \ee} + \sum_{2 \leq |n| \leq l \; \text{odd}} \frac{(r \gamma)^{|n|}}{n!} D^{(n)} f(x) \ee^n
            \nonumber\\ 
            & \phantom{\overset{\eqref{Taylor_expansion_1}}{=} \phantom{8}\,
            \| \frac{d}{\gamma} \mathbb{E} [(\dotprod{\nabla f(x)}{\gamma r \ee}}\;
            +
            \frac{R(\gamma r \ee) - R(-\gamma r \ee)}{2} ) \ee K(r)] - \nabla f(x)\|_2
            \nonumber\\ 
            &\overset{\eqref{Kernel_property}}{=} \phantom{8}\;
            \| \frac{d}{2 \gamma} \mathbb{E} \left[ \left(R(\gamma r \ee) - R(-\gamma r \ee) \right) \ee K(r) \right] \|_2 
            \nonumber\\
            &\leq \phantom{88}
            \frac{d}{2 \gamma} \mathbb{E} \left[ |R(\gamma r \ee) - R(-\gamma r \ee)|  |K(r)| \right] 
            \nonumber\\
            &\overset{\eqref{Taylor_expansion_2}}{\leq} \phantom{8}\;
            \kappa_\beta d  L_\beta \gamma^{\beta-1}. \nonumber
        \end{align*}
        Then, we can find bias square
        \begin{equation}\label{preparing_bias_square_Kernel_Theorem3}
           \| \mathbf{b}(x) \|_2^2 = \| \mathbb{E}[\mathbf{\Tilde{g}}(x,\xi,\ee)] - \nabla f(x) \|_2^2 \leq  \left(\kappa_\beta d L_\beta \gamma^{\beta-1} \right)^2 = \kappa_\beta^2 d^2 L_\beta^2 \gamma^{2(\beta-1)}. 
        \end{equation}
        
\subsection{Second moment}
        By definition \eqref{zero_order_gradient_add_noise} we have
        \begin{eqnarray} 
            \mathbb{E} \left[ \| \mathbf{\Tilde{g}}(x,\xi,e) \|^2_2\right] &=& \frac{d^2}{4 \gamma^2} \mathbb{E} \left[\|\left (f(x+\gamma r e) - f(x - \gamma r e) + \xi_1 - \xi_2 \right) K(r) e\|^2_2 \right] 
            \nonumber\\
            &=& \frac{d^2}{4\gamma^2} \mathbb{E} \left[  \left(f(x+\gamma r e) - f(x - \gamma r e) + \xi_1 - \xi_2 \right)^2 K^2(r) \| e \|_2^2 \right]  
            \nonumber\\
            &\overset{\eqref{Squared_norm_of_the_sum}}{\leq}& \frac{ d^2}{2 \gamma^2} \left( \mathbb{E} \left[ \left(f(x+\gamma r e) - f(x - \gamma r e) \right)^2 K^2(r)\right] + 4 \kappa \Tilde{\Delta}^2 \right).  \label{eq_1_for_theorem_3}
        \end{eqnarray}
        
        Using fact $\sqrt{\mathbb{E} \left[ \| e \|_q^4 \right]} \leq \min \left\{ q-1, 16 \ln d -8 \right\} d^{2/q - 1}$ to the first multiplier \eqref{eq_1_for_theorem_2} we get
        \begin{align}
            \frac{d^2}{2 \gamma^2} \mathbb{E}_{e}&  \left[\left( f(x+\gamma r e) - f(x - \gamma r e) \right)^2 \right] 
            \nonumber \\
            & = \phantom{88}
            \frac{d^2}{2 \gamma^2} \mathbb{E}_{e} \left[ \left( (f(x+\gamma r e)  - f(x - \gamma r e) \pm f(x) \pm 2 \dotprod{\nabla f(x)}{\gamma r e} \right)^2 \right]  
            \nonumber \\
            & \overset{\eqref{Squared_norm_of_the_sum}}{\leq} \phantom{8}
            \frac{3 d^2}{2 \gamma^2} \mathbb{E}_{ e} \left[ ( f(x+\gamma r e) - f(x) - \dotprod{\nabla f(x)}{\gamma r e} \right)^2
            \nonumber\\
            & \phantom{ \overset{\eqref{Squared_norm_of_the_sum}}{\leq} \, \frac{3 d^2}{2 \gamma^2} \mathbb{E}_{ e} }
            + 
            \left( f(x-\gamma r e) - f(x) - \dotprod{\nabla f(x)}{-\gamma r e} )^2 + 4 \dotprod{\nabla f(x)}{\gamma r e}^2 \right]
            \nonumber \\ 
            &\leq \phantom{88}
            \frac{3 d^2}{2 \gamma^2} \mathbb{E}_{ e} \left[ \frac{L_2^2}{2} \| \gamma r e\|_2^4 + 4 \dotprod{\nabla f(x)}{\gamma r e}^2 \right] 
            \label{eq_2_for_theorem_3} \\ 
            &\overset{\eqref{Concentration_measure}}{\leq} \phantom{8}
            \frac{3 d^2}{2 \gamma^2} \left( \frac{L_2^2 \gamma^4}{2} \mathbb{E}_{ e} \left[\| e\|_2^4 \right]  + \frac{4 \gamma^2 \| \nabla f(x) \|_2^2}{d} \right)
            \nonumber \\
            & \leq \phantom{88}
            6 d \| \nabla f(x) \|_2^2 + \frac{3  d^2 L_2^2 \gamma^2}{4}, \label{eq_3_for_theorem_3}
        \end{align}
        where \eqref{eq_2_for_theorem_3} we obtained by applying the property of the Lipschitz continuous gradient.
        
        By substituting \eqref{eq_3_for_theorem_3} into \eqref{eq_1_for_theorem_3} and using independence of $e$ and $r$ we obtain
        
        \begin{equation}\label{preparing_second_moment_Kernel_Theorem3}
            \mathbb{E} \left[ \| \mathbf{\Tilde{g}}(x,\xi,e) \|_2^2\right] \leq \kappa \left( 6 d \| \nabla f(x) \|_2^2 + \frac{3  d^2 L_2^2 \gamma^2}{4} +  \frac{2d^2 \Tilde{\Delta^2}}{\gamma^2} \right). 
        \end{equation}
        
\subsection{Convergence rate}
        From the inequalities \eqref{preparing_bias_square_Kernel_Theorem3} and \eqref{preparing_second_moment_Kernel_Theorem3} we can conclude that Assumption \ref{Assumption_3} holds with the choice
        \begin{equation}\label{parameters_Ass:3_Theorem3_Kernel}
            \sigma^2 \overset{\eqref{eq_remark_2}}{=} \mathcal{O} \left(\beta^3  d^2  L_2^2 \gamma^2 + \frac{\beta^3 d^2 \Tilde{\Delta}^2}{\gamma^2} \right), \;\;\;  M = \mathcal{O} \left( \beta^3 d \right).
        \end{equation}
        and that Assumption \ref{Assumption_4} also holds with the choice
        \begin{equation}\label{parameters_Ass:4_Theorem3_Kernel}
            m=0, \zeta^2 \overset{\eqref{eq_remark_1}}{=} \mathcal{O} \left( \beta^2 d^2 L_\beta^2 \gamma^{2(\beta-1)} \right) .
        \end{equation} 
        
        Now to find the asymptote to which the Algorithm \ref{algo: ZO-MB-SGD_algorithm} converges with the gradient approximation \eqref{zero_order_gradient_add_noise}, substitute the parameters \eqref{parameters_Ass:3_Theorem3_Kernel}, \eqref{parameters_Ass:4_Theorem3_Kernel} into the second and third terms \eqref{Itog_theorem_all} with $\eta = \mathcal{O} \left( 1/M \right)$:
        \begin{equation*}
            \mathbb{E}[f(x_N)] - f^* \leq  \frac{ \zeta^2}{2 \mu (1- m)}  + \frac{\eta L_2 \sigma^2}{2 B \mu (1- m)} = \mathcal{O} \left( \beta^2 d^2 L_\beta^2 \gamma^{2(\beta-1)} + \frac{d L_2^2 \gamma^2}{B} + \frac{d \Tilde{\Delta}^2}{B\gamma^2} \right).
        \end{equation*}
        
        Since the batch size $B$ can be taken large enough, there are no restrictions on the smoothing parameter~$\gamma$. Therefore, using the concept of the zero-order oracle \eqref{zero_order_oracle_3} Zero-order Mini-batch SGD (see Algorithm \ref{algo: ZO-MB-SGD_algorithm}) with gradient approximation \eqref{zero_order_gradient_add_noise} can achieve the desired accuracy.

\section{Proof of Theorem \ref{Theorem:combine}}\label{Appendix:Proof_combine}
   
\subsection{Kernel approximation}
    The "kernel-based" approximation of gradient has the following form \eqref{zero_order_gradient_one_point}:
    \begin{equation*}
        \mathbf{\Tilde{g}}(x,\xi,\ee) = d \frac{\Tilde{f}(x+\gamma r \ee,\xi_1) - \Tilde{f}(x - \gamma r \ee,\xi_2)}{2 \gamma} K(r) \ee,
    \end{equation*}
    where $\Tilde{f}(x,\xi)$ is defined in \eqref{zero_order_oracle_4}.
\subsection{Bias square}
        By definition \eqref{zero_order_gradient_one_point} we have
        \begin{align*}
            \| \mathbb{E}[\mathbf{\Tilde{g}}(x,&\xi,\ee)] - \nabla f(x) \|_2 = \| \frac{d}{2 \gamma} \mathbb{E} \left[ \left( \Tilde{f}(x + \gamma r \ee,\xi_1) - \Tilde{f}(x - \gamma r \ee,\xi_2) \right) \ee K(r) \right] - \nabla f(x)\|_2 
            \nonumber\\ 
            &\overset{\eqref{Stochastic_Problem},\eqref{zero_order_oracle_4}}{=}
            \| \frac{d}{2 \gamma} \mathbb{E} \left[ \left( f(x + \gamma r \ee) - f(x - \gamma r \ee) + \delta(x + \gamma r \ee) - \delta(x - \gamma r \ee) \right) \ee K(r)\right] - \nabla f(x)\|_2 
            \nonumber\\ 
            &\overset{\eqref{Taylor_expansion_1}}{=} \phantom{88}\,
            \| \frac{d}{\gamma} \mathbb{E} [(\dotprod{\nabla f(x)}{\gamma r \ee} + \sum_{2 \leq |n| \leq l \; \text{odd}} \frac{(r \gamma)^{|n|}}{n!} D^{(n)} f(x) \ee^n + \frac{R(\gamma r \ee) - R(-\gamma r \ee)}{2}   
            \nonumber\\ 
            &\phantom{\overset{\eqref{Taylor_expansion_1}}{=} \phantom{88}\,
            \| \frac{d}{\gamma} \mathbb{E} [(\dotprod{\nabla f(x)}{\gamma r \ee}} \;
            +
            \frac{\delta(x + \gamma r \ee) - \delta(x - \gamma r \ee)}{2} ) \ee K(r)] - \nabla f(x)\|_2
            \nonumber\\ 
            &\overset{\eqref{Kernel_property}}{=} \phantom{88}\,
            \| \frac{d}{2 \gamma} \mathbb{E} \left[ \left(R(\gamma r \ee) - R(-\gamma r \ee) + \delta(x + \gamma r \ee) - \delta(x - \gamma r \ee) \right) \ee K(r) \right] \|_2 
            \nonumber\\
            &\leq \phantom{888}
            \frac{d}{2 \gamma} \mathbb{E} \left[ |R(\gamma r \ee) - R(-\gamma r \ee) + \delta(x + \gamma r \ee) - \delta(x - \gamma r \ee)|  |K(r)| \right] 
            \nonumber\\
            &\overset{\eqref{Taylor_expansion_2}}{\leq} \phantom{88}\;
            \kappa_\beta d \left( L_\beta \gamma^{\beta-1} + \frac{\Delta}{\gamma} \right). \qquad /* \mbox{ the distribution of $\ee$ is symmetric  } */ \nonumber
        \end{align*}
        Then, we can find bias square
        \begin{equation}\label{preparing_bias_square_Kernel_Theorem4}
           \| \mathbf{b}(x) \|_2^2 = \| \mathbb{E}[\mathbf{\Tilde{g}}(x,\xi,\ee)] - \nabla f(x) \|_2^2  \overset{\eqref{Squared_norm_of_the_sum}}{\leq} 2 \kappa_\beta^2 d^2 L_\beta^2 \gamma^{2(\beta-1)} + 2 \frac{\kappa_\beta^2 d^2 \Delta^2}{\gamma^2}. 
        \end{equation}
        
\subsection{Second moment}
        By definition \eqref{zero_order_gradient_one_point} we have
        \begin{align} 
            \mathbb{E} \left[ \| \mathbf{\Tilde{g}}(x,\xi,e) \|^2_2\right] &= \frac{d^2}{4 \gamma^2} \mathbb{E} \left[\|\left (f(x+\gamma r e) - f(x - \gamma r e) + \delta(x + \gamma r \ee) - \delta(x - \gamma r e) + \xi_1 - \xi_2 \right) K(r) e\|^2_2 \right] 
            \nonumber\\
            &=
            \frac{d^2}{4\gamma^2} \mathbb{E} \left[  \left(f(x+\gamma r e) - f(x - \gamma r e) + 2 \delta(x + \gamma r e) + \xi_1 - \xi_2 \right)^2 K^2(r) \| e \|_2^2 \right]  
            \nonumber\\
            &\overset{\eqref{Squared_norm_of_the_sum}}{\leq}
            \frac{ d^2}{2 \gamma^2} \left( \mathbb{E} \left[ \left(f(x+\gamma r e) - f(x - \gamma r e) \right)^2 K^2(r)\right] + 4 \kappa \Delta^2 + 4 \kappa \Tilde{\Delta}^2 \right).  \label{eq_1_for_theorem_4}
        \end{align}
        
        Using fact $\sqrt{\mathbb{E} \left[ \| e \|_q^4 \right]} \leq \min \left\{ q-1, 16 \ln d -8 \right\} d^{2/q - 1}$ to the first multiplier \eqref{eq_1_for_theorem_4} we get
        \begin{align}
           \frac{d^2}{2 \gamma^2} \mathbb{E}_{e}  [( f(x&+\gamma r e) - f(x - \gamma r e) )^2 ] 
           \nonumber \\
           &= \phantom{88} \frac{d^2}{2 \gamma^2} \mathbb{E}_{e} \left[ \left( (f(x+\gamma r e)  - f(x - \gamma r e) \pm f(x) \pm 2 \dotprod{\nabla f(x)}{\gamma r e} \right)^2 \right]  
            \nonumber \\
            & \overset{\eqref{Squared_norm_of_the_sum}}{\leq} \phantom{8}
            \frac{3 d^2}{2 \gamma^2} \mathbb{E}_{ e} \left[ ( f(x+\gamma r e) - f(x) - \dotprod{\nabla f(x)}{\gamma r e} \right)^2
            \nonumber\\
            & \phantom{\leq \phantom{8}
            \frac{3 d^2}{2 \gamma^2} \mathbb{E}}
            + 
            \left( f(x-\gamma r e) - f(x) - \dotprod{\nabla f(x)}{-\gamma r e} )^2 + 4 \dotprod{\nabla f(x)}{\gamma r e}^2 \right]
            \nonumber \\ 
            &\leq \phantom{88}
            \frac{3 d^2}{2 \gamma^2} \mathbb{E}_{ e} \left[ \frac{L_2^2}{2} \| \gamma r e\|_2^4 + 4 \dotprod{\nabla f(x)}{\gamma r e}^2 \right] 
            \label{eq_2_for_theorem_4} \\ 
            &\overset{\eqref{Concentration_measure}}{\leq} \phantom{8}
            \frac{3 d^2}{2 \gamma^2} \left( \frac{L_2^2 \gamma^4}{2} \mathbb{E}_{ e} \left[\| e\|_2^4 \right]  + \frac{4 \gamma^2 \| \nabla f(x) \|_2^2}{d} \right)
            \nonumber \\
            & \leq \phantom{88}
            6 d \| \nabla f(x) \|_2^2 + \frac{3  d^2 L_2^2 \gamma^2}{4}, \label{eq_3_for_theorem_4}
        \end{align}
        where \eqref{eq_2_for_theorem_4} we obtained by applying the property of the Lipschitz continuous gradient.
        
        By substituting \eqref{eq_3_for_theorem_4} into \eqref{eq_1_for_theorem_4} and using independence of $e$ and $r$ we obtain
        
        \begin{equation}\label{preparing_second_moment_Kernel_Theorem4}
            \mathbb{E} \left[ \| \mathbf{\Tilde{g}}(x,\xi,e) \|_2^2\right] \leq \kappa \left( 6 d \| \nabla f(x,\xi) \|_2^2 + \frac{3  d^2 L_2^2 \gamma^2}{4} +  \frac{2d^2 (\Delta^2 + \Tilde{\Delta}^2)}{\gamma^2} \right). 
        \end{equation}
        
\subsection{Convergence rate}
        From the inequalities \eqref{preparing_bias_square_Kernel_Theorem4} and \eqref{preparing_second_moment_Kernel_Theorem4} we can conclude that Assumption \ref{Assumption_3} holds with the choice
        \begin{equation}\label{parameters_Ass:3_Theorem4_Kernel}
            \sigma^2 \overset{\eqref{eq_remark_2}}{=} \mathcal{O} \left(\beta^3  d^2  L_2^2 \gamma^2 + \frac{\beta^3 d^2 (\Delta^2 + \Tilde{\Delta}^2)}{\gamma^2} \right), \;\;\;  M = \mathcal{O} \left( \beta^3 d \right).
        \end{equation}
        and that Assumption \ref{Assumption_4} also holds with the choice
        \begin{equation}\label{parameters_Ass:4_Theorem4_Kernel}
            m=0, \zeta^2 \overset{\eqref{eq_remark_1}}{=} \mathcal{O} \left( \beta^2 d^2 L_\beta^2 \gamma^{2(\beta-1)} + \frac{\beta^2 d^2 \Delta^2}{\gamma^2} \right) .
        \end{equation} 
        
        Now to find the asymptote to which the Algorithm \ref{algo: ZO-MB-SGD_algorithm} converges with the gradient approximation \eqref{zero_order_gradient_one_point}, substitute the parameters \eqref{parameters_Ass:3_Theorem4_Kernel}, \eqref{parameters_Ass:4_Theorem4_Kernel} into the second and third terms \eqref{Itog_theorem_all} with $\eta = \mathcal{O} \left( 1/M \right)$:
        \begin{equation*}
            \mathbb{E}[f(x_N)] - f^* \leq  \frac{ \zeta^2}{2 \mu (1- m)}  + \frac{\eta L_2 \sigma^2}{2 B \mu (1- m)} = \mathcal{O} \left( \beta^2 d^2 L_\beta^2 \gamma^{2(\beta-1)} + \frac{\beta^2 d^2 \Delta^2}{\gamma^2} + \frac{d L_2^2 \gamma^2}{B} + \frac{d (\Delta^2 + \Tilde{\Delta}^2)}{B \gamma^2} \right).
        \end{equation*}
        
        Since $B$ can be taken as large, the first two terms are responsible for the asymptote. We find the optimal smoothing parameter~$\gamma$ that minimizes the first two terms:
        \begin{eqnarray}
             \mathbb{E}[f(x_N)] - f^* \leq \beta^2 d^2 \gamma^{2(\beta-1)} + \frac{\beta^2 d^2 \Delta^2}{\gamma^2} = \mathcal{O} \left(d^{2} \Delta^{\frac{2(\beta-1)}{\beta}}\right), \label{assimptota_theorem4_Kernel}
        \end{eqnarray}
        where $\gamma =  \Delta^{1/\beta}$  is optimal smoothing parameter.
        Then from \eqref{assimptota_theorem4_Kernel} we can find the maximum noise level, assuming that~$d^{2} \Delta \leq \varepsilon$, for $\varepsilon > 0$ then we have
        \begin{equation*}
             \Delta = \mathcal{O}\left( \varepsilon^{\frac{\beta}{2(\beta-1)}} d^{\frac{-\beta}{\beta-1}} \right) .
        \end{equation*}

\section{Gaussian smoothing approach for Theorem \ref{Theorem:1}}\label{Gaussian_Proof_Theorem1}
    \subsection{Definition Gaussian approximation}
    Let $\uu \sim \mathcal{N}(0,1)$ is a random Gaussian vector, $\gamma>0$ is smoothing parameter then Gaussian smoothing approximation has the following~form:
    \begin{equation}
        \label{Gaussian_approximation_grad_Theorem_1}
        \Tilde{\gg}(x,\uu) = \frac{\Tilde{f}(x + \gamma \uu) - \Tilde{f}(x)}{\gamma} \uu,
    \end{equation}
    where $\Tilde{f}$ is defined in \eqref{zero_order_oracle_1}.
    \subsection{Bias square}
    By definition Gaussian approximation \eqref{Gaussian_approximation_grad_Theorem_1} we have
    \begin{align*}
            \left\| \mathbb{E}[\Tilde{\gg}(x,\uu)] - \nabla f(x) \right\|_2  &\phantom{88} = \phantom{88} \left\| \frac{1}{\gamma} \mathbb{E} \left[ \left( \Tilde{f}(x + \gamma \uu) - \Tilde{f}(x) \right) \uu \right] - \nabla f(x)\right\|_2 
            \nonumber\\ 
            &\phantom{88} \overset{\eqref{zero_order_oracle_1}}{=} \phantom{88}
            \left\| \frac{1}{\gamma} \mathbb{E} \left[ \left( f(x + \gamma \uu) - f(x) + \delta(x + \gamma \uu) - \delta(x) \right) \uu \right] - \nabla f(x)\right\|_2  
            \nonumber\\ 
            &\phantom{88} = \phantom{88}
            \left\| \frac{1}{\gamma} \mathbb{E}_{\uu} \left[ \left( f(x + \gamma \uu) - f(x) - \gamma \dotprod{\nabla f(x)}{\uu} + \delta(x + \gamma \uu) - \delta(x) \right) \uu  \right]\right\|_2
            \nonumber\\ 
            & \phantom{88} \leq \phantom{88}
            L_2 \gamma \mathbb{E}_{\uu} \left\| \uu \right\|_2^3 + \frac{\Delta}{\gamma} \mathbb{E}_{\uu} \left\|\uu \right\|_2 \qquad /* \mbox{ the distribution of $\uu$ is symmetric  } */ \nonumber \\
            &\overset{\eqref{2_moment_for_u}, \eqref{p_moment_for_u}}{\leq} 
            L_2 \gamma d^{3/2} + \frac{\Delta}{\gamma} d^{1/2}.
        \end{align*}
        Then we can find the bias square:
        \begin{equation}\label{preparing_bias_square_Gaussian_Theorem1}
            \|\bb(x) \|_2^2 = \left\| \mathbb{E}[\mathbf{\Tilde{g}}(x,\uu)] - \nabla f(x) \right\|^2_2 \leq \left(L_2 \gamma d^{3/2} + \frac{\Delta}{\gamma} d^{1/2} \right)^2 \overset{\eqref{Squared_norm_of_the_sum}}{\leq} 2 L_2^2 \gamma^2 d^{3} + 2 \frac{\Delta^2}{\gamma^2} d.
        \end{equation}
        
        \subsection{Second moment}
        By definition \eqref{Gaussian_approximation_grad_Theorem_1} we have
        \begin{align}
            \mathbb{E} \| \Tilde{\gg}(x,\uu) \|^2_2 & \phantom{888} = \phantom{888} \mathbb{E} \left[ \| \frac{1}{\gamma} \left( \Tilde{f}(x + \gamma \uu) - \Tilde{f}(x) \right) \uu \|^2_2 \right] 
            \nonumber \\
            & \phantom{888}\overset{\eqref{zero_order_oracle_1}}{\leq} \phantom{888}
            \frac{1}{\gamma^2} \mathbb{E} \left[ \left( f(x + \gamma \uu) - f(x) + \delta(x + \gamma \uu) - \delta(x) \right)^2 \| \uu\|^2_2 \right] \nonumber \\
            &\phantom{888} \leq  \phantom{888}
            \frac{1}{\gamma^2}\mathbb{E} \left[ \left( f(x + \gamma \uu) - f(x) \pm \gamma \dotprod{\nabla f(x)}{\uu} + \delta(x + \gamma \uu) - \delta(x) \right)^2 \| \uu\|^2_2 \right]
            \nonumber \\
            &\phantom{888} \overset{\eqref{Squared_norm_of_the_sum}}{\leq} \phantom{888}
            \frac{3}{\gamma^2}  L_2^2 \gamma^4 \mathbb{E}_\uu [ \| \uu\|^6_2] + 3 \mathbb{E}_\uu \left[  \dotprod{\nabla f(x)}{\uu}^2 \| \uu\|^2_2 \right] + \frac{3 \Delta^2}{\gamma^2} \mathbb{E}_{\uu} \left[\| \uu\|^2_2 \right] 
            \nonumber \\
            &\overset{\eqref{2_moment_for_u},\eqref{p_moment_for_u}, \eqref{Fact_for_second_moment_Gaussian}}{\leq}
            3 L_2^2 \gamma^2 d^3 + 3 d \| \nabla f(x) \|_2^2 + \frac{3 d \Delta^2}{\gamma^2}. \label{preparing_second_moment_Gaussian_Theorem1}
        \end{align}
        
\subsection{Convergence rate}
        From the inequalities \eqref{preparing_bias_square_Gaussian_Theorem1} and \eqref{preparing_second_moment_Gaussian_Theorem1} we can conclude that Assumption \ref{Assumption_3} holds with the choice
        \begin{equation}\label{parameters_Ass:3_Theorem1_Gaussian}
            \sigma^2 = \mathcal{O} \left( \gamma^2 d^3 + \frac{d \Delta^2}{\gamma^2} \right), \;\;\;  M = \mathcal{O} \left( d \right).
        \end{equation}
        and that Assumption \ref{Assumption_4} also holds with the choice
        \begin{equation}\label{parameters_Ass:4_Theorem1_Gaussian}
            m=0, \zeta^2 = \mathcal{O} \left(\gamma^2 d^{3} + \frac{\Delta^2}{\gamma^2} d \right) .
        \end{equation} 
        
        Now to find the asymptote to which the counterpart of Algorithm \ref{algo: ZO-MB-SGD_algorithm} converges with the approximation \eqref{Gaussian_approximation_grad_Theorem_1}, substitute the parameters \eqref{parameters_Ass:3_Theorem1_Gaussian}, \eqref{parameters_Ass:4_Theorem1_Gaussian} into the second and third terms~\eqref{Itog_theorem_all}~with~$\eta = \mathcal{O} \left( 1/M \right)$:
        \begin{equation*}
            \mathbb{E}[f(x_N)] - f^* \leq  \frac{ \zeta^2}{2 \mu (1- m)}  + \frac{\eta L_2 \sigma^2}{2 B \mu (1- m)} = \mathcal{O} \left( \gamma^2 d^{3} + \frac{\Delta^2}{\gamma^2} d + \frac{\gamma^2 d^2}{B} + \frac{\Delta^2}{B\gamma^2} \right).
        \end{equation*}
        
        Since $B$ can be taken as large, the first two terms are responsible for the asymptote. We find the optimal smoothing parameter~$\gamma$ that minimizes the first two terms:
        \begin{eqnarray}
             \mathbb{E}[f(x_N)] - f^* \leq \gamma^2 d^{3} + \frac{\Delta^2}{\gamma^2} d
            = \mathcal{O} \left(\Delta d^2 \right), \label{assimptota_theorem1_Gaussian}
        \end{eqnarray}
        where $\gamma = \Delta^{1/2} d^{-1/2}$  is optimal smoothing parameter.
        
        Then from \eqref{assimptota_theorem1_Gaussian} we can find the maximum noise level, assuming that $\Delta d^2 \leq \varepsilon$, for $\varepsilon > 0$~then~we~have
        \begin{equation*}
             \Delta = \mathcal{O}\left( \frac{\varepsilon}{d^2} \right) .
        \end{equation*}

\section{Gaussian smoothing approach with two-point feedback}\label{Gaussian_Proof_Theorem2}
    \subsection{Definition Gaussian approximation}
    Let $\uu \sim \mathcal{N}(0,1)$ is a random Gaussian vector, $\gamma>0$ is smoothing parameter then Gaussian smoothing approximation has the following~form:
    \begin{equation}
        \label{Gaussian_approximation_grad_Theorem_2}
        \Tilde{\gg}(x,\xi,\uu) = \frac{\Tilde{f}(x + \gamma \uu, \xi) - \Tilde{f}(x, \xi)}{\gamma} \uu,
    \end{equation}
    where $\Tilde{f}$ is defined in \eqref{zero_order_oracle_2}.
    \subsection{Bias square}
    By definition Gaussian approximation \eqref{Gaussian_approximation_grad_Theorem_2} we have
    \begin{align*}
            \|\bb(x) \|_2 & = \phantom{88,} \| \mathbb{E}[\Tilde{\gg}(x,\xi,\uu)] - \nabla f(x) \|_2    
            \nonumber\\ 
            &  = \phantom{88} \left\| \frac{1}{\gamma} \mathbb{E} \left[ \left( \Tilde{f}(x + \gamma \uu, \xi) - \Tilde{f}(x, \xi) \right) \uu \right] - \nabla f(x)\right\|_2 
            \nonumber\\ 
            &  \!\! \overset{\eqref{zero_order_oracle_2}}{=} \phantom{88}
            \left\| \frac{1}{\gamma} \mathbb{E}_{\uu} \left[ \mathbb{E}_{\xi} \left[ \left( f(x + \gamma \uu, \xi) \right] - \mathbb{E}_{\xi}\left[f(x, \xi)\right] + \delta(x + \gamma \uu) - \delta(x) \right) \uu \right] - \nabla f(x)\right\|_2  
            \nonumber\\ 
            &  \overset{\eqref{Stochastic_Problem}}{=} \phantom{88}
            \left\| \frac{1}{\gamma} \mathbb{E}_{\uu} \left[  \left( f(x + \gamma \uu) - f(x) + \delta(x + \gamma \uu) - \delta(x) \right) \uu \right] - \nabla f(x)\right\|_2  
            \nonumber\\ 
            & = \phantom{88}
            \left\| \frac{1}{\gamma} \mathbb{E}_{\uu} \left[ \left( f(x + \gamma \uu) - f(x) - \gamma \dotprod{\nabla f(x)}{\uu} + \delta(x + \gamma \uu) - \delta(x) \right) \uu  \right]\right\|_2
            \nonumber\\ 
            &  \leq  \phantom{88}
            L_2 \gamma \mathbb{E}_{\uu} \left\| \uu \right\|_2^3 + \frac{\Delta}{\gamma} \mathbb{E}_{\uu} \left\|\uu \right\|_2 \qquad /* \mbox{ the distribution of $\uu$ is symmetric  } */ \nonumber \\
            &\!\!\!\!\! \overset{\eqref{2_moment_for_u}, \eqref{p_moment_for_u}}{\leq}
            L_2 \gamma d^{3/2} + \frac{\Delta}{\gamma} d^{1/2}.
        \end{align*}
        Then we can find the bias square:
        \begin{equation}\label{preparing_bias_square_Gaussian_Theorem2}
            \|\bb(x) \|_2^2 = \left\| \mathbb{E}[\mathbf{\Tilde{g}}(x, \xi,\uu)] - \nabla f(x) \right\|^2_2 \leq \left(L_2 \gamma d^{3/2} + \frac{\Delta}{\gamma} d^{1/2} \right)^2 \overset{\eqref{Squared_norm_of_the_sum}}{\leq} 2 L_2^2 \gamma^2 d^{3} + 2 \frac{\Delta^2}{\gamma^2} d.
        \end{equation}
        
        \subsection{Second moment}
        By definition \eqref{Gaussian_approximation_grad_Theorem_2} we have
        \begin{align}
            \mathbb{E} \| \Tilde{\gg}(x, \xi,\uu) \|^2_2 &\;  = \phantom{888} \mathbb{E} \left[ \| \frac{1}{\gamma} \left( \Tilde{f}(x + \gamma \uu, \xi) - \Tilde{f}(x, \xi) \right) \uu \|^2_2 \right] 
            \nonumber \\
            &  \overset{\eqref{zero_order_oracle_2}}{\leq} \phantom{888}
            \frac{1}{\gamma^2} \mathbb{E} \left[ \left( f(x + \gamma \uu, \xi) - f(x, \xi) + \delta(x + \gamma \uu) - \delta(x) \right)^2 \| \uu\|^2_2 \right] \nonumber \\
            & \; \leq  \phantom{888}
            \frac{1}{\gamma^2}\mathbb{E} \left[ \left( f(x + \gamma \uu, \xi) - f(x, \xi) \pm \gamma \dotprod{\nabla f(x, \xi)}{\uu} + \delta(x + \gamma \uu) - \delta(x) \right)^2 \| \uu\|^2_2 \right]
            \nonumber \\
            & \overset{\eqref{Squared_norm_of_the_sum}}{\leq} \phantom{888}
            \frac{3}{\gamma^2}  L_2^2 \gamma^4 \mathbb{E}_\uu [ \| \uu\|^6_2]  + 3 \mathbb{E}_\xi \left[ \mathbb{E}_\uu \left[  \dotprod{\nabla f(x, \xi)}{\uu}^2 \| \uu\|^2_2 \right]\right] + \frac{3 \Delta^2}{\gamma^2} \mathbb{E}_{\uu} \left[\| \uu\|^2_2 \right] 
            \nonumber \\
            &\!\!\!\!\!\!\!\! \overset{\eqref{2_moment_for_u},\eqref{p_moment_for_u}, \eqref{Fact_for_second_moment_Gaussian}}{\leq}\;
            3 L_2^2 \gamma^2 d^3 + 3 d \| \nabla f(x) \|_2^2 + \frac{3 d \Delta^2}{\gamma^2}. \label{preparing_second_moment_Gaussian_Theorem2}
        \end{align}
        
\subsection{Convergence rate}
        From the inequalities \eqref{preparing_bias_square_Gaussian_Theorem2} and \eqref{preparing_second_moment_Gaussian_Theorem2} we can conclude that Assumption \ref{Assumption_3} holds with the choice
        \begin{equation}\label{parameters_Ass:3_Theorem2_Gaussian}
            \sigma^2 = \mathcal{O} \left( \gamma^2 d^3 + \frac{d \Delta^2}{\gamma^2} \right), \;\;\;  M = \mathcal{O} \left( d \right),
        \end{equation}
        and that Assumption \ref{Assumption_4} also holds with the choice
        \begin{equation}\label{parameters_Ass:4_Theorem2_Gaussian}
            m=0, \zeta^2 = \mathcal{O} \left(\gamma^2 d^{3} + \frac{\Delta^2}{\gamma^2} d \right) .
        \end{equation} 
        
        Now to find the asymptote to which the counterpart of Algorithm \ref{algo: ZO-MB-SGD_algorithm} converges with the approximation \eqref{Gaussian_approximation_grad_Theorem_2}, substitute the parameters \eqref{parameters_Ass:3_Theorem2_Gaussian}, \eqref{parameters_Ass:4_Theorem2_Gaussian} into the second and third terms~\eqref{Itog_theorem_all}~with~$\eta = \mathcal{O} \left( 1/M \right)$:
        \begin{equation*}
            \mathbb{E}[f(x_N)] - f^* \leq  \frac{ \zeta^2}{2 \mu (1- m)}  + \frac{\eta L_2 \sigma^2}{2 B \mu (1- m)} = \mathcal{O} \left( \gamma^2 d^{3} + \frac{\Delta^2}{\gamma^2} d + \frac{\gamma^2 d^2}{B} + \frac{\Delta^2}{B\gamma^2} \right).
        \end{equation*}
        
        Since $B$ can be taken as large, the first two terms are responsible for the asymptote. We find the optimal smoothing parameter~$\gamma$ that minimizes the first two terms:
        \begin{eqnarray}
             \mathbb{E}[f(x_N)] - f^* \leq \gamma^2 d^{3} + \frac{\Delta^2}{\gamma^2} d
            = \mathcal{O} \left(\Delta d^2 \right), \label{assimptota_theorem2_Gaussian}
        \end{eqnarray}
        where $\gamma = \Delta^{1/2} d^{-1/2}$  is optimal smoothing parameter.
        
        Then from \eqref{assimptota_theorem2_Gaussian} we can find the maximum noise level, assuming that $\Delta d^2 \leq \varepsilon$, for $\varepsilon > 0$~then~we~have
        \begin{equation*}
             \Delta = \mathcal{O}\left( \frac{\varepsilon}{d^2} \right) .
        \end{equation*}

\section{Gaussian smoothing approach for Theorem \ref{Theorem:add_noise}}\label{Gaussian_Proof_Theorem3}
    \subsection{Definition Gaussian approximation}
    Let $\uu \sim \mathcal{N}(0,1)$ is a random Gaussian vector, $\gamma>0$ is smoothing parameter then Gaussian smoothing approximation has the following~form:
    \begin{equation}
        \label{Gaussian_approximation_grad_Theorem_3}
        \Tilde{\gg}(x,\xi,\uu) = \frac{\Tilde{f}(x + \gamma \uu, \xi_1) - \Tilde{f}(x,\xi_2)}{\gamma} \uu,
    \end{equation}
    where $\Tilde{f}$ is defined in \eqref{zero_order_oracle_3}.
    \subsection{Bias square}
    By definition Gaussian approximation \eqref{Gaussian_approximation_grad_Theorem_3} we have
        \begin{align}
            \left\| \mathbb{E}[\Tilde{\gg}(x,\xi,\uu)] - \nabla f(x) \right\|_2 &= \phantom{88} \left\| \frac{1}{\gamma} \mathbb{E} \left[ \left( \Tilde{f}(x + \gamma \uu, \xi_1) - \Tilde{f}(x, \xi_2) \right) \uu \right] - \nabla f(x)\right\|_2 
            \nonumber\\ 
            & \! \! \overset{\eqref{zero_order_oracle_3}}{=} \phantom{88}
            \left\| \frac{1}{\gamma} \mathbb{E} \left[ \left( f(x + \gamma \uu) - f(x) + \xi_1 - \xi_2 \right) \uu \right] - \nabla f(x)\right\|_2  
            \nonumber\\ 
            &= \phantom{88}
            \left\| \frac{1}{\gamma} \mathbb{E}_{\uu} \left[ \left( f(x + \gamma \uu) - f(x) - \gamma \dotprod{\nabla f(x)}{\uu} + \xi_1 - \xi_2 \right) \uu  \right]\right\|_2
            \nonumber\\ 
            & \leq  \phantom{88}
            L_2 \gamma \mathbb{E}_{\uu} \left\| \uu \right\|_2^3 
            \label{property_delta_Theorem3_Gaussian} \\
            & \!\!\!\!\! \overset{\eqref{2_moment_for_u}, \eqref{p_moment_for_u}}{\leq}
            L_2 \gamma d^{3/2}.
            \nonumber
        \end{align}
        where we receive \eqref{property_delta_Theorem3_Gaussian} using that $\mathbb{E}[\xi_1 \uu] = 0$ and $\mathbb{E}[\xi_2 \uu] = 0$. 
        
        Then we can find the bias square:
        \begin{equation}\label{preparing_bias_square_Gaussian_Theorem3}
            \|\bb(x) \|_2^2 = \left\| \mathbb{E}[\mathbf{\Tilde{g}}(x,\xi,\uu)] - \nabla f(x) \right\|^2_2 \leq \left(L_2 \gamma d^{3/2} \right)^2 = L_2^2 \gamma^2 d^{3}.
        \end{equation}
        
        \subsection{Second moment}
        By definition \eqref{Gaussian_approximation_grad_Theorem_3} we have
        \begin{align}
            \mathbb{E} \| \Tilde{\gg}(x, \xi,\uu) \|^2_2 &\; = \phantom{888} \mathbb{E} \left[ \| \frac{1}{\gamma} \left( \Tilde{f}(x + \gamma \uu, \xi_1) - \Tilde{f}(x, \xi_2) \right) \uu \|^2_2 \right] 
            \nonumber \\
            &\overset{\eqref{zero_order_oracle_3}}{\leq} \phantom{888}
            \frac{1}{\gamma^2} \mathbb{E} \left[ \left( f(x + \gamma \uu) - f(x) + \xi_1 - \xi_2 \right)^2 \| \uu\|^2_2 \right] \nonumber \\
            &\; \leq \phantom{888}
            \frac{1}{\gamma^2}\mathbb{E} \left[ \left( f(x + \gamma \uu) - f(x) \pm \gamma \dotprod{\nabla f(x)}{\uu} + \xi_1 - \xi_2 \right)^2 \| \uu\|^2_2 \right]
            \nonumber \\
            &\overset{\eqref{Squared_norm_of_the_sum}}{\leq} \phantom{888}
            \frac{3}{\gamma^2}  L_2^2 \gamma^4 \mathbb{E}_\uu \left[ \| \uu\|^6_2 \right]  + 3     \mathbb{E}_\uu \left[ \dotprod{\nabla f(x)}{\uu}^2 \| \uu\|^2_2 \right]  + \frac{3 \Tilde{\Delta}^2}{\gamma^2} \mathbb{E} \left[\| \uu\|^2_2 \right] 
            \nonumber \\
            & \!\!\!\!\!\!\!\!\! \overset{\eqref{2_moment_for_u},\eqref{p_moment_for_u}, \eqref{Fact_for_second_moment_Gaussian}}{\leq} \;
            3 L_2^2 \gamma^2 d^3 + 3 d \| \nabla f(x) \|_2^2 + \frac{3 d \Tilde{\Delta}^2}{\gamma^2}. \label{preparing_second_moment_Gaussian_Theorem3}
        \end{align}
        
\subsection{Convergence rate}
        From the inequalities \eqref{preparing_bias_square_Gaussian_Theorem3} and \eqref{preparing_second_moment_Gaussian_Theorem3} we can conclude that Assumption \ref{Assumption_3} holds with the choice
        \begin{equation}\label{parameters_Ass:3_Theorem3_Gaussian}
            \sigma^2 = \mathcal{O} \left( \gamma^2 d^3 + \frac{d \Delta^2}{\gamma^2} \right), \;\;\;  M = \mathcal{O} \left( d \right).
        \end{equation}
        and that Assumption \ref{Assumption_4} also holds with the choice
        \begin{equation}\label{parameters_Ass:4_Theorem3_Gaussian}
            m=0, \zeta^2 = \mathcal{O} \left(\gamma^2 d^{3} \right) .
        \end{equation} 
        
        Now to find the asymptote to which the counterpart of Algorithm \ref{algo: ZO-MB-SGD_algorithm} converges with the approximation \eqref{Gaussian_approximation_grad_Theorem_3}, substitute the parameters \eqref{parameters_Ass:3_Theorem3_Gaussian}, \eqref{parameters_Ass:4_Theorem3_Gaussian} into the second and third terms~\eqref{Itog_theorem_all}~with~$\eta = \mathcal{O} \left( 1/M \right)$:
        \begin{equation*}
            \mathbb{E}[f(x_N)] - f^* \leq  \frac{ \zeta^2}{2 \mu (1- m)}  + \frac{\eta L_2 \sigma^2}{2 B \mu (1- m)} = \mathcal{O} \left( \gamma^2 d^{3} + \frac{\gamma^2 d^2}{B} + \frac{\Delta^2}{B \gamma^2} \right).
        \end{equation*}
        
         Since the batch size $B$ can be taken large enough, there are no restrictions on the smoothing parameter~$\gamma$. Therefore, using the concept of the zero-order oracle \eqref{zero_order_oracle_3} the gradient-free counterpart of ZO-MB-SGD (see Algorithm \ref{algo: ZO-MB-SGD_algorithm}) with gradient approximation \eqref{Gaussian_approximation_grad_Theorem_3} can achieve the desired accuracy.

\section{Gaussian smoothing approach for Theorem \ref{Theorem:combine}}\label{Gaussian_Proof_Theorem4}
    \subsection{Definition Gaussian approximation}
    Let $\uu \sim \mathcal{N}(0,1)$ is a random Gaussian vector, $\gamma>0$ is smoothing parameter then Gaussian smoothing approximation has the following~form:
    \begin{equation}
        \label{Gaussian_approximation_grad_Theorem_4}
        \Tilde{\gg}(x,\xi,\uu) = \frac{\Tilde{f}(x + \gamma \uu, \xi_1) - \Tilde{f}(x, \xi_2)}{\gamma} \uu,
    \end{equation}
    where $\Tilde{f}$ is defined in \eqref{zero_order_oracle_4}.
    \subsection{Bias square}
    By definition Gaussian approximation \eqref{Gaussian_approximation_grad_Theorem_4} we have
        \begin{align*}
            \|\bb(x) \|_2 & = \phantom{88} \left\| \mathbb{E}[\Tilde{\gg}(x,\xi,\uu)] - \nabla f(x) \right\|_2 
            \nonumber\\ 
            &= \phantom{88}
            \left\| \frac{1}{\gamma} \mathbb{E} \left[ \left( \Tilde{f}(x + \gamma \uu, \xi) - \Tilde{f}(x, \xi) \right) \uu \right] - \nabla f(x)\right\|_2 
            \nonumber\\ 
            &\!\! \overset{\eqref{zero_order_oracle_4}}{=} \phantom{8}\;\,
            \left\| \frac{1}{\gamma} \mathbb{E}_{\uu} \left[  \left( f(x + \gamma \uu) - f(x) + \delta(x + \gamma \uu) - \delta(x) + \xi_1 - \xi_2 \right) \uu \right] - \nabla f(x)\right\|_2  
            \nonumber\\ 
            &= \phantom{88}
            \left\| \frac{1}{\gamma} \mathbb{E}_{\uu} \left[ \left( f(x + \gamma \uu) - f(x) - \gamma \dotprod{\nabla f(x)}{\uu} + \delta(x + \gamma \uu) - \delta(x) \right) \uu  \right]\right\|_2
            \nonumber\\ 
            & \leq  \phantom{88}
            L_2 \gamma \mathbb{E}_{\uu} \left\| \uu \right\|_2^3 + \frac{\Delta}{\gamma} \mathbb{E}_{\uu} \left\|\uu \right\|_2 \qquad /* \mbox{ the distribution of $\uu$ is symmetric  } */ \nonumber \\
            & \!\!\!\!\! \overset{\eqref{2_moment_for_u}, \eqref{p_moment_for_u}}{\leq}
            L_2 \gamma d^{3/2} + \frac{\Delta}{\gamma} d^{1/2}.
        \end{align*}
        Then we can find the bias square:
        \begin{equation}\label{preparing_bias_square_Gaussian_Theorem4}
            \|\bb(x) \|_2^2 = \left\| \mathbb{E}[\mathbf{\Tilde{g}}(x, \xi,\uu)] - \nabla f(x) \right\|^2_2 \leq \left(L_2 \gamma d^{3/2} + \frac{\Delta}{\gamma} d^{1/2} \right)^2 \overset{\eqref{Squared_norm_of_the_sum}}{\leq} 2 L_2^2 \gamma^2 d^{3} + 2 \frac{\Delta^2}{\gamma^2} d.
        \end{equation}
        
        \subsection{Second moment}
        By definition \eqref{Gaussian_approximation_grad_Theorem_4} we have
        \begin{align}
            \mathbb{E} \| \Tilde{\gg}(x,& \xi,\uu) \|^2_2
            \nonumber \\
            &\, = \; \phantom{888}
            \mathbb{E} \left[ \| \frac{1}{\gamma} \left( \Tilde{f}(x + \gamma \uu, \xi) - \Tilde{f}(x, \xi) \right) \uu \|^2_2 \right] 
            \nonumber \\
            &\overset{\eqref{zero_order_oracle_2}}{\leq} \phantom{888}
            \frac{1}{\gamma^2} \mathbb{E} \left[ \left( f(x + \gamma \uu) - f(x) + \delta(x + \gamma \uu) - \delta(x) + \xi_1 - \xi_2 \right)^2 \| \uu\|^2_2 \right] \nonumber \\
            &\, \leq \, \phantom{888}
            \frac{1}{\gamma^2}\mathbb{E} \left[ \left( f(x + \gamma \uu) - f(x) \pm \gamma \dotprod{\nabla f(x)}{\uu} + \delta(x + \gamma \uu) - \delta(x) + \xi_1 - \xi_2 \right)^2 \| \uu\|^2_2 \right]
            \nonumber \\
            &\overset{\eqref{Squared_norm_of_the_sum}}{\leq} \phantom{888}
            \frac{3}{\gamma^2}  L_2^2 \gamma^4 \mathbb{E}_\uu [ \| \uu\|^6_2]  + 3  \mathbb{E}_\uu \left[  \dotprod{\nabla f(x, \xi)}{\uu}^2 \| \uu\|^2_2 \right] + \frac{3 ( \Delta^2 + \Tilde{\Delta}^2)}{\gamma^2} \mathbb{E}_{\uu} \left[\| \uu\|^2_2 \right] 
            \nonumber \\
            &\!\!\!\!\!\!\!\! \overset{\eqref{2_moment_for_u},\eqref{p_moment_for_u}, \eqref{Fact_for_second_moment_Gaussian}}{\leq}
            3 L_2^2 \gamma^2 d^3 + 3 d \| \nabla f(x) \|_2^2 + \frac{3 d ( \Delta^2 + \Tilde{\Delta}^2)}{\gamma^2}. \label{preparing_second_moment_Gaussian_Theorem4}
        \end{align}
        
\subsection{Convergence rate}
        From the inequalities \eqref{preparing_bias_square_Gaussian_Theorem4} and \eqref{preparing_second_moment_Gaussian_Theorem4} we can conclude that Assumption \ref{Assumption_3} holds with the choice
        \begin{equation}\label{parameters_Ass:3_Theorem4_Gaussian}
            \sigma^2 = \mathcal{O} \left( \gamma^2 d^3 + \frac{d (\Delta^2 + \Tilde{\Delta}^2)}{\gamma^2} \right), \;\;\;  M = \mathcal{O} \left( d \right).
        \end{equation}
        and that Assumption \ref{Assumption_4} also holds with the choice
        \begin{equation}\label{parameters_Ass:4_Theorem4_Gaussian}
            m=0, \zeta^2 = \mathcal{O} \left(\gamma^2 d^{3} + \frac{\Delta^2}{\gamma^2} d \right) .
        \end{equation} 
        
        Now to find the asymptote to which the counterpart of Algorithm \ref{algo: ZO-MB-SGD_algorithm} converges with the approximation \eqref{Gaussian_approximation_grad_Theorem_4}, substitute the parameters \eqref{parameters_Ass:3_Theorem4_Gaussian}, \eqref{parameters_Ass:4_Theorem4_Gaussian} into the second and third terms~\eqref{Itog_theorem_all}~with~$\eta = \mathcal{O} \left( 1/M \right)$:
        \begin{equation*}
            \mathbb{E}[f(x_N)] - f^* \leq  \frac{ \zeta^2}{2 \mu (1- m)}  + \frac{\eta L_2 \sigma^2}{2 B \mu (1- m)} = \mathcal{O} \left( \gamma^2 d^{3} + d \frac{\Delta^2}{\gamma^2}  + \frac{\gamma^2 d^2}{B} + \frac{ \Delta^2 + \Tilde{\Delta}^2}{B\gamma^2} \right).
        \end{equation*}
        
        Since $B$ can be taken as large, the first two terms are responsible for the asymptote. We find the optimal smoothing parameter~$\gamma$ that minimizes the first two terms:
        \begin{eqnarray}
             \mathbb{E}[f(x_N)] - f^* \leq \gamma^2 d^{3} + d \frac{\Delta^2}{\gamma^2} 
            = \mathcal{O} \left(d^2 \Delta  \right), \label{assimptota_theorem4_Gaussian}
        \end{eqnarray}
        where $\gamma = \Delta^{1/2} d^{-1/2}$  is optimal smoothing parameter.
        
        Then from \eqref{assimptota_theorem4_Gaussian} we can find the maximum noise level, assuming that $\Delta d^2 \leq \varepsilon$, for $\varepsilon > 0$~then~we~have
        \begin{equation*}
             \Delta = \mathcal{O}\left( \frac{\varepsilon}{d^2} \right) .
        \end{equation*}




\end{document}